\documentclass[11p]{amsart}
\usepackage{amsfonts}
\usepackage{amssymb}
\usepackage{amscd}

\input xy
\xyoption{all}
\title[$0$-Schur algebras]{A geometric realisation of $0$-Schur and $0$-Hecke algebras}
\author{Bernt Tore Jensen and Xiuping Su}

\newtheorem{theorem}{Theorem}[section]
\newtheorem{theorem2}{Theorem}
\newtheorem{lemma}[theorem]{Lemma}
\newtheorem{proposition}[theorem]{Proposition}

\newtheorem{example}[theorem]{Example}

\newtheorem{corollary}[theorem]{Corollary}

\newcommand{\gl}{\mathrm{GL}}

\newcommand{\rad}{\mathrm{rad}}

\renewcommand{\ker}{\mathrm{ker}}
\newcommand{\im}{\mathrm{im}}

\newcommand{\dvector}{{\underline{\mathrm{dim}}}}

\newcommand{\lra}{\longrightarrow}

\newcommand{\ra}{\rightarrow}
\newcommand{\sdp}{\times\kern-.2em\vrule height1.1ex depth-.05ex}
\newcommand{\epi}{\lra \kern-.8em\ra}

\normalbaselines
\begin{document}

\begin{abstract}
We define a new product on orbits of pairs of flags in a vector space, using open orbits
in certain varieties of pairs of flags. This new product defines 
an associative $\mathbb{Z}$-algebra, denoted by $G(n,r)$. We 
show that $G(n,r)$ is a geometric realisation of the $0$-Schur algebra 
$S_0(n, r)$ over $\mathbb{Z}$, 
which is the $q$-Schur algebra $S_q(n,r)$ at $q=0$. 
We view a pair of flags as a pair of projective resolutions for a quiver
of type $\mathbb{A}$ with linear orientation, and study $q$-Schur algebras 
from this point of view. This allows us to understand
the relation between $q$-Schur algebras and Hall algebras and 
construct bases of $q$-Schur algebras, which are used in the proof
of the main results.  Using the geometric realisation, we 
construct idempotents and multiplicative bases for $0$-Schur algebras.  
We also give a geometric realisation of $0$-Hecke algebras and a  
presentation of the $q$-Schur algebra over a base ring where $q$ is not 
invertible. 
\end{abstract}

\maketitle

\section*{Introduction}

Let $k$ be a finite or an algebraically closed field. When $k$ is
finite, we denote by $q$  the number of elements in $k$.
Let $n\geq 1$ and $r\geq 1$ be integers and let $\mathcal{F}$ denote
the variety of partial $r$-step flags in an $n$-dimensional vector space
$V$. The general linear group $\gl(V)$ acts on $\mathcal{F}$ by
change of basis on $V$ and the orbits under this action are denoted
by $\mathcal{F}_{\underline d}$, where ${\underline d}$ is a decomposition 
of the integer $n$ into $r$ parts.

Beilinson, Lusztig and MacPherson \cite{BLM} (see also Du \cite{Duj}
and Green \cite{GreenR}) construct the $q$-Schur algebra $S_q(n,r)$
on the basis of $\gl(V)$-orbits in
$\mathcal{F}\times \mathcal{F}$, where $\gl(V)$ acts diagonally on
$\mathcal{F}\times \mathcal{F}$ by change of basis on $V$. This
algebra can be defined using the pair of maps
$$\xymatrix{\mathcal{F}\times \mathcal{F} \times \mathcal{F}
\ar^{\pi}[d] & \ar^{\Delta}[r] && (\mathcal{F}\times \mathcal{F})
\times (\mathcal{F}\times \mathcal{F})  \\ \mathcal{F}\times \mathcal{F}}$$
where $\Delta(f,f',f'')=((f,f'),(f',f''))$ and $\pi(f,f',f'')=(f,f'')$, and the
structure constant in front of an orbit $[g,g']=\gl(V)(g,g')$ in the product
of $[f,f']$ and $[f',f'']$ is the number of elements in the set
$$\pi^{-1}(g,g')\cap \Delta^{-1}([f,f']\times [f'f'']).$$ The structure constants
are given by polynomials evaluated at $q$ for any finite field.  
Beilinson, Lusztig and MacPherson construct
a limit of these finite dimensional algebras to give a geometric realisation
of the quantised enveloping algebra of $gl_n$.

In this paper we take the point of view that a pair of flags is a pair of projective 
resolutions for a quiver of type $\mathbb{A}$ with linear orientation. 
We show that a pair of flags and its corresponding pair of projective resolutions 
uniquely determine each other. We then construct bases and
prove results that will be used in the subsequent parts of the paper. 

The main goal of this paper is to give a geometric realisation of 
$0$-Schur algebras, which are $q$-Schur algebras at $q=0$. 
We define a new algebra $G(n,r)$ on the same basis as $S_q(n, r)$ by defining the
product of $[f,f']$ and $[f',f'']$ to be the unique open orbit (see Section
\ref{genmult}) in $$\pi\Delta^{-1}([f,f']\times [f',f'']),$$ and
obtain the following main result.   
\begin{theorem2}\label{Main}
As $\mathbb{Z}$-algebras, 
$G(n,r)$ is isomorphic to $S_q(n,r)_{q=0}$.
\end{theorem2}

We remark that the definition of the new product is similar to the one defined by Reineke 
\cite{Reineke} for Hall algebras and the main result generalises the main result in \cite{Su}.

To prove Theorem \ref{Main}, we first give a presentation of $G(n,r)$ 
using quivers and relations and then use the presentation to prove the isomorphism 
in the theorem.  We remark that 
Deng and Yang \cite{DY} have independently given a similar presentation
for $S_0(n,r)$, using a different approach. 

Among applications of the geometric realisation, we  obtain 
a multiplicative basis for $S_0(n,r)$ and 
using open orbits we construct a block in $S_0(n,r)$ isomorphic
to a matrix algebra and families of idempotents. Also, we give a presentation of 
$S_q(n, r)$ over a base ring where $q$ is not invertible. 
In the special case $n=r$, we  obtain a geometric
realisation of the $0$-Hecke algebra $H_0(n)$  

The paper is organised as follows. In Section 1 we
recall the construction of Beilinson, Lusztig and MacPherson and in Section 2
we describe orbits in $\mathcal{F}\times \mathcal{F}$ using representations
of linear quivers of type $\mathbb{A}$.  In Section 3 we recall the definition
of the positive and negative parts of the $q$-Schur algebras and their
relationship to the Hall algebra, and use this description to construct a basis
of the $q$-Schur algebra in Section 4. In Section 5 we give a describe 
quantised Schur algebras using quivers and relations. 
We define the generic algebra in Section 6 and show that it is isomorphic 
to the $0$-Schur-algebra in Section 7. In Section 8 we consider the
degeneration order of orbits in $\mathcal{F}\times \mathcal{F}$, and use
open orbits to construct idempotents
for the $0$-Schur algebra in Section 9. Finally, we discuss 
$0$-Hecke algebras in Section 10.

\section{Flag varieties and $q$-Schur algebras} \label{firstsec}

In this section we fix notation and recall some definitions and results
of Beilinson, Lusztig and MacPherson \cite{BLM} on $q$-Schur
algebras. We also recall the definition of the $0$-Schur algebra.

Let $n,r\geq 1$ be integers and $V$ an $r$-dimensional vector space
over a field $k$. Denote by $\mathcal{F}$ the set of all $n$-steps flags
$$f:\{0\}=V_0\subseteq V_1 \subseteq \cdots\subseteq
V_n=V$$ in $V$. The general linear group $\gl(V)$ acts on $\mathcal{F}$
by change of basis on $V$.

For any $f\in \mathcal{F}$, define $d_i=\dim V_i-\dim V_{i-1}$ for
$i=1, \cdots, n$. Then $$\underline d=d_1+\cdots +d_n$$ is a decomposition
of $r$ into $n$ parts. Two flags $f$ and $g$ are isomorphic, i.e. they
are in the same $\gl(V)$-orbit, if and only if they have the same decomposition.
Let $D(n,r)$ denote the set of all decompositions of $r$ in $n$ parts, and
let $\mathcal{F}_{\underline d}\subseteq \mathcal{F}$ denote the orbit
corresponding to $\underline d\in D(n,r)$.

Let $\gl(V)$ act diagonally on $\mathcal{F}\times \mathcal{F}$, i.e. by change
of basis on $V$.
Given a pair of flags $(f,f')\in \mathcal{F}\times \mathcal{F},$
$$f:\{0\}=V_0\subseteq V_1 \subseteq \cdots \subseteq V_n=V$$ and
$$f':\{0\}=V'_0\subseteq V'_1 \subseteq \cdots \subseteq V'_n=V,$$ define
a matrix $A=A(f,f')=(A_{ij})$, with $$A_{ij}=\dim({V_{i-1}+V_i \cap
V'_{j}})-\dim({V_{i-1}+V_i \cap V'_{j-1}}).$$ This defines a bijection
between the $\gl(V)$-orbits in ${\mathcal F}\times {\mathcal F}$ and
matrices of non-negative integers with entries which sum to $r$. We
denote the $\gl(V)$-orbit of $(f, f')$ by $[f, f']$ and by $e_A$ if we
want to emphasize the matrix $A=A(f,f')$. Two pairs of flags $(f,f')$
and $(g,g')$ are isomorphic if they belong to the same $\gl(V)$-orbit
and we write $(f,f')\simeq (g,g')$ in this case.

Let $e_A, e_{A'}, e_{A''}\in \mathcal{F}\times \mathcal{F}/\gl(V)$ and
$(f_1, f_2)\in e_{A''}$. Let $$S(A,A',A'')=\{f\in \mathcal{F}\mid (f_1, f)\in
e_A, (f, f_2)\in e_{A'} \}.$$ Following Proposition 1.1 in \cite{BLM}, there
exists a polynmial $g_{A, A', A''} \in \mathbb{Z}[q]$, such that
$$g_{A, A', A''}(q)= |S(A,A',A'')|,$$ for all finite fields.
The projection $\mathcal{F}\times \mathcal{F} \times
\mathcal{F}\rightarrow \mathcal{F}$ onto the middle factor maps
$\Delta^{-1}(e_A\times e_{A'})\cap \pi^{-1}(f_1,f_2)$ bijectively
onto $S(A,A',A'')$, and so these two sets have the same cardinality.

Recall \cite{BLM}  that the $q$-Schur algebra $S_q(n, r)$ is the free
$\mathbb{Z}[q]$-module with basis $\mathcal{F}\times\mathcal{F}/
\gl(V)$, and associative multiplication given by $$e_{A}e_{A'} =
\sum_{e_{A''}\in \mathcal{F}\times \mathcal{F}/\gl(V)}g_{A, A', A''}
e_{A''}.$$

Although, in general it is difficult to compute the polynomial $g_{A, A', A''}$,
the following lemma from \cite{BLM}, dealing with special $A$ and $A'$, gives clear
multiplication rules. Also, Deng and Yang give a recursive formula of
 $g_{A, A', A''}$ using Hall polynomials \cite{DengYang2}.

Let $$[m]=\frac{q^m-1}{q-1}=q^{m-1}+\cdots+q + 1$$ for
 $m\in\mathbb{N}$ and let $E_{i,j}$ denote the $(i,j)'$th elementary
matrix.

\begin{lemma} \cite{BLM} \label{Lemma3.2}
Assume that $1\leq h<n$. Let $e_A\subseteq \mathcal{F}_{\underline e}
\times \mathcal{F}_{\underline f}$. Assume that $e_B\subseteq
\mathcal{F}_{\underline d}\times \mathcal{F}_{\underline e}$ and
$e_C\subseteq \mathcal{F}_{\underline d'}\times \mathcal{F}_{\underline e}$
such that  $B-E_{h, h+1}$, $C-E_{h+1, h}$ are diagonal matrices. Then
the following multiplication formuli hold in $S_q(n,r),$
$$
e_Be_A=\sum_{\{p| A_{h+1, p}>0\}} q^{\sum_{j>p}A_{h,j}}
[A_{h,p}+1]e_{X},
$$

$$
e_Ce_A=\sum_{\{p| A_{h, p}>0\}}q^{\sum_{j<p}A_{h+1, j}}
[A_{h+1,p}+1]
e_{Y},
$$
\noindent where $X=A+E_{h,p}-E_{h+1,p}$ and $Y=A-E_{h,p}+E_{h+1,p}$.
\end{lemma}

The classical Schur algebra $$S(n,r)=S(n,r)\otimes_{\mathbb{Z}[q]}
\mathbb{Z}[q]/(q-1)$$ is obtained by evaluating the structure constants
$g_{A,A',A''}$ at $q=1$, and the $0$-Schur algebra $$S_0(n,r)=S_q(n,r)
\otimes_{\mathbb{Z}[q]} \mathbb{Z}[q]/(q)$$ is obtained by evaluating
the structure constants $g_{A,A',A''}$ at $q=0$.

\section{Representations of linear quivers} \label{linearquiver}

In this section we describe orbits of pairs of flags  using representations
of a linear quiver $\Lambda=\Lambda(n)$ of type $\mathbb{A}_n$,
$$\Lambda(n):\xymatrix{1 \ar[r] & 2 \ar[r] & \cdots \ar[r] & n}.$$

A representation $X$ of $\Lambda$ consists of vector spaces $X_i$,
$i=1,\cdots,n$, and linear maps $X_\alpha:X_i\rightarrow X_{i+1}$,
$\alpha:i\rightarrow i+1$, $i=1,\cdots,n-1$. A homomorphism
$h:X\rightarrow Y$ between two representations $X$ and $Y$ is a collection
of linear maps $h_i:X_i\rightarrow Y_i$, satisfying $h_{i+1}X_\alpha=
Y_{\alpha}h_i$. A homomorphism $h$ is an isomorphism if  $h_i$ is a
linear isomorphism for all $i=1,\cdots,n$, and we write $X\simeq Y$ if
$X$ and $Y$ are isomorphic. A direct sum of representations is a
direct sum of vector spaces and maps. A nonzero representation is
indecomposable if it is not isomorphic to a direct sum of non-zero
representations.

The dimension vector of $X$ is denoted by $\dvector X=(\dim X_i)_i$.
If we fix each vector space $X_i$ we may parameterise representations
of $\Lambda$ by the vector space of all maps $(X_\alpha)_\alpha$ on
which $\prod_i \gl(X_i)$ acts by change of basis, such that the orbits
correspond to isomorphism classes of representations with dimension
vector $\dvector X$.

Let $M_{ij}$, for $j\geq i$, be the indecomposable representation supported
on the interval of vertices $[i,j]=\{i,\cdots,j\}$, with vector spaces in the support
equal to $k$ and all non-zero maps equal to the identity.
Any representation $M$ decomposes uniquely, up to isomorphism,
as $$M\simeq \bigoplus_{i,j}(M_{ij})^{d_{ij}}$$ for non-negative integers
$d_{ij}$. For each vertex $i$, let $S_i=M_{ii}$ and $P_i=M_{in}$ be the
simple and indecomposable projective representation at $i$, respectively.
A representation $P$ is projective if and only if each map $P_\alpha$ is
injective.

For each decomposition $\underline d\in D(n,r)$, let $P(\underline d)$ be
the projective representation defined by $$P(\underline d)=\bigoplus_i
{(P_i)}^{d_i}, P(\underline d)_n=V.$$ By taking images of the maps
$P(\underline d)_i\rightarrow P(\underline d)_n=V$ we get a $n$-step flag in
$\mathcal{F}_{\underline d}$. We will view any projective representation
$P(d)$ as a flag in $\mathcal{F}_d$.

An $n$-step flag in $V$ is a projective representation of $\Lambda$, with
maps equal to the the inclusions $V_i\subseteq V_{i+1}$. Two flags are
isomorphic if and only if they are isomorphic as representations. If $f$ is
a flag in $U$ and $f'$ is a flag in $U'$ then $f\oplus f'$ denotes the flag in
$V=U\oplus U'$ with vector space at each vertex $i$ equal to $U_i\oplus U_i'$.

A pair of flags $(g,f)$ with $g\subseteq f$ can be viewed as a projective
resolution $$0 \rightarrow g \rightarrow f \rightarrow f/g \rightarrow 0.$$
If $(f_1, f_2), (f_1', f_2')\in \mathcal{F}\times \mathcal{F}$, with
$f_1\subseteq f_2$ and $f_1'\subseteq f_2'$, then $(f_1, f_2)\simeq (f_1', f_2')$
if and only if $f_2/f_1\simeq f_2'/f_1'$ and $f_2\simeq f_2'$. This fact
generalises to arbitrary pairs, and this is the main lemma of this section.

\begin{lemma}\label{orbit1}
Let $(f_1, f_2), (f_1', f_2')\in \mathcal{F}\times \mathcal{F}$. The following
are equivalent.
\begin{itemize}

\item[i)] $(f_1, f_2)\simeq (f_1', f_2')$.

\item[ii)] $(f_1+f_2)/f_i\simeq (f_1'+f_2')/f_i'$ for $i=1,2$, and
$f_1+f_2\simeq f_1'+f_2'$.

\item[iii)]  $(f_i, f_1+f_2)\simeq (f_i', f_1'+f_2')$ for $i=1,2$.

\end{itemize}
\end{lemma}

\begin{proof}
The implication from $i)$ to $ii)$ is trivial.

We prove that $ii)$ implies $i)$.
By $ii)$, $f_i/(f_1\cap f_2)\simeq f_i'/(f_1'\cap f_2')$, and
$f_1\cap f_2\simeq f'_1 \cap f '_2$.

Let $g_i:f_i/f_1\cap f_2\rightarrow f_i'/f_1'\cap f_2'$ be isomorphisms.
Consider the following diagram,
$$\xymatrix{f_1+f_2\ar[rr]^{\pi \hspace{45pt}} \ar[d]^{\exists h}
&&(f_1/f_1\cap f_2)\oplus (f_2/f_1\cap f_2)\ar[d]^{\tiny \left(
\begin{matrix} g_1 & 0 \\ 0 &g_2\end{matrix}\right)} \\ f'_1+f'_2
\ar[rr]^{\pi'\hspace{42pt} }  && (f'_1/f'_1\cap f'_2)
\oplus (f'_2/f'_1\cap f'_2),}$$ where $\pi$ and $\pi'$ are natural
projections. Since $f_1+f_2$ and $f_1'+f_2'$ are isomorphic
projective representations, there is an isomorphism $h$ such
that the above diagram commutes. Thus $h(f_1)\subseteq (\pi')^{-1}
(f_1'/f_1'\cap f_2')=f_1' $. Therefore $h(f_1)=f_1' $. Similarly
$h(f_2)=f_2'$. Hence $(f_1, f_2)$ and $(f_1', f_2')$ are in the
same orbit. This proves $i)$

$iii)$ is a reformulation of $ii)$. This finishes the proof.
\end{proof}

We mention that it is possible to restate the lemma by replacing the
inclusions $f_i\subseteq f_1+f_2$ by the inclusions $f_1\cap f_2\subseteq f_i$.
Also, the condition $f_1+f_2\simeq f_1'+f_2'$ in part ii) can be replaced by
$f_1\simeq f_1'$ and $f_2\simeq f_2'$.

We give some consequences of the lemma. The lemma shows that
pairs of flags in $\mathcal{F}\times \mathcal{F}$ are determined up to
isomorphism by triples $(\underline d,[M],[N])$,  where $\underline d\in
D(n,r)$ and $[M]$, $[N]$ are isomorphism classes of representations of
$\Lambda$ with a surjection $P(\underline d)\rightarrow M\oplus N$.
Conversely, given a triple $(\underline d,[M],[N])$ with $\underline d\in D(n,r)$
and a surjection $\phi:P(\underline d)\rightarrow M\oplus N$ we construct a
corresponding pair $(g_1,g_2)$ as follows. Recall that a surjective homomorphism
$\psi:P\rightarrow M$ from a projective representation $P$ is a projective cover
if $\ker\psi\subseteq \rad P$ where $\rad P$ denotes the Jacobson radical of $P$.
Decompose $$P(\underline d)=f_1 \oplus f_2 \oplus c$$ such that $\phi_{|f_1}$
is a projective cover of $M$ and $\phi_{|f_2}$ is a projective cover of $N$.
Then the kernel $\ker\phi$ decomposes as $$\ker\phi=f'_1\oplus f'_2\oplus c$$
where $f'_i=\ker\phi_{|f_i}$. Now let $$(g_1,g_2)=(f_1\oplus f'_2\oplus c,
f'_1\oplus f_2\oplus c).$$ Then $g_1+g_2=f_1\oplus f_2\oplus c, g_1
\cap g_2=f'_1\oplus f'_2\oplus c$ and $f'_1\subseteq \rad f_1, f'_2
\subseteq \rad f_2$, $g_1+g_2/g_2\simeq g_1/g_1\cap g_2\simeq M$ and
 $g_1+g_2/g_1\simeq g_2/g_1\cap g_2\simeq N$.

Given a pair of flags $(f,f')\in \mathcal{F}\times \mathcal{F}$, we
have the following descripton of the matrix $A$ corresponding to the
orbit $e_A=[f,f']$.

\begin{lemma} \label{Lemma2.2}
Let $e_A=[f,f']$.  Then
\begin{itemize}
\item[i)] if $i<j$, then $A_{ij}$ is the multiplicity of $M_{i-1,j}$
as a direct summand in $f/f\cap f'$,
\item[ii)] if $i>j$, then $A_{ij}$ is the multiplicity of $M_{j,i-1}$
as a direct summand in $f'/f\cap f'$, and
\item[iii)] $A_{ii}$ is the multiplicity of $M_{in}\rightarrow M_{in}$
as a contractible summand in the projective resolution
$f\cap f' \subseteq f+f'$.
\end{itemize}
\end{lemma}
\begin{proof}
Let $(f,f')\in \mathcal{F}\times \mathcal{F},$ and $f''=f\cap f'$ be given
by $f: \{0\}=V_0\subseteq V_1 \subseteq \cdots\subseteq V_n,$
$f':\{0\}=V'_0\subseteq V'_1 \subseteq \cdots \subseteq V'_n,$ and
$f'':\{0\}=V''_0\subseteq V''_1 \subseteq \cdots \subseteq V''_n.$
We have  $$A_{ij}=\dim\frac{V_{i-1}+V_i \cap V'_{j}}{V_{i-1}+V_i
\cap V'_{j-1}}.$$

Then for  $i<j$ let $e_B=[f,f'']$.  We have
$$B_{ij}=\dim\frac{V_{i-1}+V_i \cap V''_{j}}{V_{i-1}+V_i
\cap V''_{j-1}} = \dim\frac{V_{i-1}+V_i \cap V'_{j}}{V_{i-1}+V_i
\cap V'_{j-1}}$$ $$ = \dim \frac{V_{i-1}+V_i\cap V_j\cap V'_j}{V_{i-1}+
V_i\cap V_{j-1}\cap V'_{j-1}}=\dim\frac{V_{i-1}+V_i\cap V'_{j}}
{V_{i-1}+V_i\cap V'_{j-1}}=A_{ij}.$$ But $B_{ij}$ is the multiplicity of
$M_{i-1,j}$ as a direct summand in $f/f''$, which proves $i)$.

The proof of $ii)$ is similar.

For the diagonal of $A$, we have $$A_{ii}=\frac{V_{i-1}+V_i\cap V'_{i}}
{V_{i-1}+V_i\cap V'_{i-1}}.$$ Write $V_i\cap V'_{i}=U\oplus W$ where
$U\subseteq V_{i-1}$ and $W\cap V_{i-1}=0$. Now $W=W_1\oplus W_2$,
$W_1\subseteq V_i\cap V'_{i-1}+V_{i-1}$ and $W_2\cap (V_i\cap
V'_{i-1}+V_{i-1})=0$. Therefore $A_{ii}=\dim W_2$ which is the multiplicity
of $M_{in}\rightarrow M_{in}$ as a contractible summand in the projective
resolution $f\cap f'\rightarrow f+f'$.  This proves iii).
\end{proof}

\section{The non-negative $q$-Schur algebra}

In this section we describe the non-negative part of a $q$-Schur
algebra as a Hall algebra of projective resolutions of representations
of the linear quiver $\Lambda=\Lambda(n)$, defined in Section
\ref{linearquiver}. We also include some easy lemmas on the
computation of Hall numbers for the linear quiver which are needed
in subsequent sections. For deeper connections between Hall numbers
and the structure constants of the $q$-Schur algebra, please see
\cite{DengYang2}.

An orbit $[f,f']\in \mathcal{F}\times \mathcal{F}/\gl(V)$ with
$f'\subseteq f$ decomposes as $$[f,f']=[c\oplus g, c\oplus g']$$
where $f/f'\simeq g/g'$ and $g'\subseteq \rad g$. Such an orbit
can be viewed as a choice of a projective resolution of $g/g'$
obtained by adding a contractible summand $c$ to the minimal
projective resolution $0\rightarrow g' \rightarrow g \rightarrow
g/g' \rightarrow 0$.

The non-negative $\mathbb{Z}[q]$-subalgebra $S^+_q(n,r)$ is the
subalgebra of $S_q(n,r)$ with basis consisting of all orbits $[f,f']$ with
$f'\subseteq f$. Similarly, the non-positive $q$-Schur algebra $S^-_q(n,r)$ has
the corresponding basis of all orbits $[f',f]$ with $f'\subseteq f$.

Recall that the Hall number $h^L_{MN}$ defined by Ringel 
\cite{RingelBanach} is the number of submodules $X\subseteq L$
satisfying $X\simeq N$ and $L/X\simeq M$.

\begin{lemma} \label{Hallnumbers}
Let  $f_1\supseteq f_2\supseteq f_3$ be flags and let $e_{A}=[f_1, f_2]$,
$e_{A'}=[f_2, f_3]$, $e_{A''}=[f_1',f_3']$ with $f'_1\supseteq f'_3$,
$f_1'\simeq f_1$ and $f_3'\simeq f_3$,  $M=f_1/f_2$, $N=f_2/f_3$ and
$L=f_1'/f_3'$. Then $$g_{A,A',A''}=h_{MN}^{L}. $$
\end{lemma}
\begin{proof}
Let $U=\{X\subseteq L\mid X\simeq N, L/X\simeq M\}.$ 
We will define two mutually inverse maps between $U$ and $S(A,A',A'')$.
Given $f'_2\in S(A,A',A'')$, we have 
the following commutative diagram of short exact sequences

$$
\xymatrix{
0\ar[r] &f'_3\ar[r] \ar[d]& f_3'\ar[d] \ar[r]&0
\ar[r]\ar[d]&0\\ 0\ar[r] & f_2'\ar[r]\ar[d] & f_1'
\ar[r]\ar[d]^{\pi}&f_1'/f_2'\ar[r]\ar[d] &0\\
0\ar[r] & f_2'/f_3'\ar[r] &L\ar[r]& f_1'/f_2'\ar[r]&0,
}
$$
with $f_1/f_2\simeq f_1'/
f_2'\simeq M$ and $f_2/f_3\simeq f_2'/f_3'\simeq N$. 

Define maps $S(A,A',A'')\rightarrow U$ by $f_2'\mapsto \pi(f_2')$ and
$U\rightarrow S(A,A',A'')$ by $X\mapsto \pi^{-1}(X)$. It is easy
to check that these two maps are mutually inverse, and so the
equality follows since $h_{MN}^{L}=|U|$.
\end{proof}

Denote the (non-twisted) Ringel-Hall algebra \cite{RingelBanach} by
$H_q(\Lambda)$. That is, $H_q(\Lambda)$ is the free $\mathbb{Z}[q]$-module
with basis isomorphism classes $[M]$ of representations of $\Lambda$ and
multiplication $$[M][N]=\sum_Lh^L_{MN}[L].$$ Mapping modules to
choices of projective resolutions induces an algebra homomorphism
$$\Theta^+: H_q(\Lambda) \ra S^+_q(n, r) \mbox{ defined by } [M]
\mapsto \sum_{\{[f,f']\mid f'\subseteq f, f/f'\simeq M\}}[f, f']$$ with kernel
spanned by those $[M]$ with the number of indecomposable direct
summands bigger than $r$ \cite{GreenR}. There is a similar map $\Theta^-:
H_q(\Lambda) \ra S^-_q(n, r)$.

As a conseqence, we have the following special cases of Corollary 4.5 in
\cite{BLM} (see also Proposition 14.1 in \cite{DDPW}). The assumptions
are as in Lemma \ref{Hallnumbers}.

\begin{corollary}
We have $$g_{A+D,  A'+D, A''+D}=h^L_{M,N},$$
for any diagonal matrix $D=\mathrm{diag}(d_1, \cdots, d_n)$ with $d_i\geq 0$.
\end{corollary}

This yields a different proof of Theorem 14.27 in \cite{DDPW},
which we restate here as follows.

\begin{theorem}
Hall algebra $H_q(\Lambda)$ is isomorphic to the algebra with basis all formal
sums $$\{\sum_{f/g\cong M}[f, g]\mid M \in \mathrm{mod} kQ\},$$ over all
decompositions of $r$ into $n$ parts and for all $r\geq 1$, and the multiplication
as in $S_q(n,r)$.
\end{theorem}

As before, let $M_{ij}$ be the indecomposable representation of $\Lambda$
supported on the interval $[i,j]$. Let $M_{ij}\leq M_{i'j'}$ if $j<j'$ or
$i\leq i'$ if $j=j'$. This is a total order on the indecomposable representations
of $\Lambda$. Observe that if $M_{ij}\leq M_{i'j'}$ then
$Ext^1(M_{i'j'},M_{ij})=0$, but that the converse is not true in general.
We state a lemma which follows easily from Lemma \ref{Hallnumbers} and the 
corresponding computations in the Ringel-Hall algebra.

\begin{lemma} \label{posmult}
Suppose flags $f\supseteq g$ with $f/g=\oplus_{ij} M_{ij}^{m_{ij}}$. Then
there exists a filtration $f=f_n \supset f_{n-1} \supset \cdots \supset f_1=g \supset 0$
with indecomposable factors $f_i/f_{i-1}\leq
f_{i+1}/f_i$ for all $i$, such that $$[f_n,f_{n-1}]\cdots [f_2,f_1]=[f_n, f_1]\prod
[m_{ij}]!$$ and $m_{ij}$ is the multiplicity of $M_{ij}$ as a subfactor
in the filtration.
\end{lemma}

\section{Bases of $S_q(n,r)$}

In this section we describe a basis for $S_q(n,r)$ using the non-negative
and non-positive subalgebras defined in the previous section. Let $$\mathcal{B}
=\{[f_1, f_1+f_2][f_1+f_2,f_2]\mid [f_1,f_2]\in \mathcal{F}\times
\mathcal{F}/\gl(V) \}.$$ Lemma \ref{orbit1} shows that the map
$\mathcal{F}\times \mathcal{F}/\gl(V)\ra \mathcal{B}$ given by
$$[f_1, f_2]\mapsto [f_1, f_1+f_2][f_1+f_2, f_2]$$ is a well defined
surjection of sets.
We shall see that $\mathcal{B}$ is a $\mathbb{Z}[q]$-basis
of $S_q(n,r)$. Note that there is a similar basis $\mathcal{B}'$ of $S_q(n,r)$ consisting
of elements of the form $[f_1,f_1\cap f_2][f_1\cap f_2,f_2]$.

\begin{lemma}\label{bases1}
Let $(f_1, f_2)\in \mathcal{F}\times\mathcal{F}$ and let $e_{A}=
[f_1, f_1+f_2]$ and $e_{A'}=[f_1+f_2, f_2]$. Then $$[f_1, f_1+f_2]
[f_1+f_2, f_2]=[f_1, f_2]+\sum_{\{e_{A''}=[f_1',f_2']\mid f_1'+f_2'
\subsetneq f_1+f_2\}}g_{A,A',A''}[f_1', f_2'].$$
\end{lemma}
\begin{proof}
Suppose that $[f_1',f_2']$  is one of the terms with a non-zero coefficient
in the sum $$[f_1, f_1+f_2][f_1+f_2, f_2]=\sum g_{A,A',A''}[f_1', f_2'].$$
Then there exists an $f\in \mathcal{F}$ such
that $(f_1', f)\simeq (f_1, f_1+f_2)$ and $(f, f_2')\simeq
(f_1+f_2, f_2)$. Thus $f_1', f_2'\subseteq f$ and so $f_1'+f_2'
\subseteq f$. Note that if $f_1'+f_2'=f$, then $(f_1', f)=(f_1', f_1'+f_2')
\simeq (f_1, f_1+f_2) \mbox{ and }$ $(f, f_2')=(f_1'+f_2', f_2')
\simeq (f_1+f_2, f_2).$ Therefore $(f_1'+f_2')/f_i'\simeq
(f_1+f_2)/f_i$ for $i=1, 2$. By Lemma \ref{orbit1}, $(f_1', f_2')
\simeq (f_1, f_2)$. Moreover $g_{A,A',A''}=1$ for $e_{A}=[f_1, f_2]$.
The lemma follows.
\end{proof}

There is a similar
formula for the product $[f_1, f_1\cap f_2]
[f_1\cap f_2, f_2]$.

\begin{theorem}\label{2basis}
The set $\mathcal{B}$ is a $\mathbb{Z}[q]$-basis of $S_q(n, r)$.
\end{theorem}

\begin{proof}
By induction on the size of $f_1+f_2$ and Lemma \ref{bases1},
any basis element $[f_1, f_2]$ can be written as a
$\mathbb{Z}[q]$-linear combination of elements in $\mathcal{B}$.
This proves that $\mathcal{B}$ spans $S_q(n,r)$ as a
$\mathbb{Z}[q]$-module. Then $\mathcal{B}$ is in bijection with
the basis $\mathcal{F}\times\mathcal{F}/\gl(V)$, which proves that
$\mathcal{B}$ is a basis of $S_q(n,r)$.
\end{proof}

Each basis element in $\mathcal{B}$ has a decomposition
$$[c\oplus f'_1 \oplus f_2, c \oplus f_1\oplus f_2][c\oplus f_1 \oplus
f_2,c\oplus f_1\oplus f'_2]$$ where $f_1'\subseteq \rad f_1$ and
$f_2'\subseteq \rad f_2$.

\begin{lemma}
Let $f_1\subseteq g \supseteq f_2$ be flags in $\mathcal{F}$.
Then $(f_1,g)\simeq (h_1,h_1+h_2)$ and $(g,f_2)\simeq (h_1+h_2,h_2)$
for a pair of flags $h_1,h_2\in \mathcal{F}$,  if and only if there is a
surjective map $g \rightarrow g/f_1 \oplus g/f_2$.
\end{lemma}
\begin{proof}
Assume that $\phi:g \rightarrow g/f_1 \oplus g/f_2$ is surjective.
Let $h_1=\phi^{-1}(g/f_2)$ and $h_2=\phi^{-1}(g/f_1)$. Then
$(h_1, g)\simeq (f_1, g)$ and $(g, h_2) \simeq (g, f_2)$ and
$g=\phi^{-1}(g/f_1\oplus g/f_2)=h_1+h_2$.

The converse follows since the map $\pi:h_1+h_2\rightarrow
(h_1+h_2)/h_1\oplus (h_1+h_2)/h_2$ is surjective.
\end{proof}

In particular the lemma shows that a surjective map $g \rightarrow
g/f_1 \oplus g/f_2$ implies $[f_1,g][g,f_2]\in \mathcal{B}$. The following
example shows that the converse is not true.

\begin{example}
Let $n=2$, $r=2$, $V=span\{x_1,x_2\}$, $f_i:0\subseteq kx_i$ for $i=1,2$, and $g:V\subseteq V$. Then $$[f_1,g][g,f_2]=[f_1,f_2]+[f_1,f_1]=[f_1,f_1+f_2][f_1+f_2,f_2]\in \mathcal{B},$$ with no surjective map $g\rightarrow g/f_1\oplus g/f_2$. In this case $g\not \simeq f_1+f_2$.
\end{example}

\section{Quiver and relations for $q$-Schur algebras}

In this section we present an algebra using quivers and binomial relations,
which will be shown to be the $0$-Schur algebra in Section \ref{real}. This 
will lead to a presentation of the $q$-Schur over a 
base ring where $q$ is not invertible. Also, following from the relations,  
the $0$-Schur algebra has a multiplicative basis of paths which will 
be constructed geometrically in Section \ref{genmult} and \ref{real}. 

As before, let $D(n, r)$ be the set of decompositions of $r$ into $n$
parts. For $\underline d\in D(n,r)$, we define $\underline d+\alpha_i$ by
$$(\underline d+\alpha_i)_j=\left\{\begin{matrix} d_i+1,
\mbox{  if } i=j,  \\  d_j,  \mbox{  if } i\neq j.\end{matrix}\right.$$

Let $\Sigma(n,r)$ be the quiver with vertices $K_{\underline d}$ and
arrows $E_{i,\underline d}$ and $F_{i,\underline d}$,
$$\xymatrix{K_{\underline d+\alpha_i-\alpha_{i+1}} \;
\ar@/^15pt/[rr]^{F_{i,\underline d+\alpha_i-\alpha_{i+1}}} & \;  \;  \; \; 
\;  \;  \;  \;  \; \;  \;  \;  & \ar@/^15pt/[ll]^{E_{i,\underline d} } \;  \; \;
\;  \;  \;   K_{\underline d}}$$ where $\underline d, \underline d+{\alpha_i}-
\alpha_{i+1} \in D(n,r)$. The vertices can be drawn on a simplex,
where the vertices $K_{\underline d}$ with $d_i=0$ for some $i=0$
are on the boundary, and vertices $K_{\underline  d}$ with
$d_i\neq 0$ for all $i$ are in the interior of the simplex.

To simplify our formulas we define $$E_{i,\underline d}=0 \mbox{ if }
\left\{\begin{matrix}\underline d \not \in D(n,r), \\ i=n, \mbox{ or }
\\ d_{i+1}=0 \end{matrix}\right.,$$

$$F_{i,\underline d}=0 \mbox{ if }
\left\{\begin{matrix}\underline d \not \in D(n,r), \\ i=n, \mbox{ or }
\\ d_{i}=0 \end{matrix}\right.,$$ $$K_{\underline d}=0 \mbox{ if }
\underline d\not \in D(n,r),$$ $$E_i =\sum_{\underline d}
E_{i,\underline d}\mbox{ and } F_i=\sum_{\underline d} F_{i,\underline d}.$$

For a commutative ring $R$, denote by $R\Sigma(n,r)$ the path $R$-algebra
of $\Sigma(n,r)$, which is the free $R$-module with basis all paths in
$\Sigma(n,r)$, and multiplication given by composition of paths. The vertices
$K_d$ form an orthogonal set of idempotents in $R\Sigma(n,r)$ and the
composition of two paths $p$ and $q$ is $pq$, if $q$ ends where $p$
starts, and zero otherwise. Recall that a relation in $R\Sigma(n,r)$ is a
$R$-linear combination of paths with common starting and ending vertex
$$\rho = \sum_i r_ip_i, r_i\in R, p_i \mbox{ a path.}$$

Let $I(n,r)\subseteq \mathbb{Z}[q]\Sigma(n,r)$ be the ideal generated
by the relations
$$P_{ij, \underline d}=K_{\underline d+p_{ij}}P_{ij}K_{\underline d},$$
$$N_{ij, \underline d}=K_{\underline d-p_{ij}}N_{ij}K_{\underline d}, \mbox{ and }$$ $$C_{ij,\underline d}=
K_{\underline d+\alpha_{i}+\alpha_{j+1}-\alpha_{i+1}-\alpha_j}
C_{ij}K_{\underline d},$$ where $$p_{ij}=\left\{ \begin{matrix} 2\alpha_i+
\alpha_j-2\alpha_{i+1}-\alpha_{j+1}, \mbox{ if } |i-j| = 1, \\ \alpha_i+
\alpha_j-\alpha_{i+1}-\alpha_{j+1}, \mbox{ if } |i-j| > 1; \end{matrix}\right.$$
$$P_{ij}=\left\{  \begin{matrix} E_i^2E_j-(q+1)
E_iE_jE_i+q E_jE^2_{i} \mbox{ for } i=j-1, \\ qE_i^2E_j-(q+1)
E_iE_jE_i+ E_jE^2_{i} \mbox{ for } i=j+1, \\ E_{i}E_{j}-E_{j} E_{i},
\mbox{ otherwise;}\end{matrix}\right.$$ $$N_{ij}=\left\{  \begin{matrix}
qF_i^2F_j-(q+1)F_iF_jF_i+ F_jF^2_{i} \mbox{ for }  i=j-1, \\ F_i^2F_j-(q+1)
F_iF_jF_i+ qF_jF^2_{i} \mbox{ for } i=j+1, \\ F_{i}F_{j}-F_{j} F_{i},
\mbox{ otherwise; } \end{matrix}\right. $$ and $$C_{ij}=E_{i}F_{j}-
F_{j}E_{i}-\delta_{ij}\sum_{\underline d}\frac{q^{d_i}-q^{d_{i+1}}}{q-1}
K_{\underline d}.$$

Let $$e_{i,\underline d}=[f,f'],
f_{i,\underline d+\alpha_i-\alpha_{i+1}}=[f',f] \mbox{ and }
k_{\underline d}=[h,h]$$ where $(f,f')\in \mathcal{F}\times
\mathcal{F}$ with $f'\subseteq f$, $f/f'\simeq S_i$, and $f',h\in
\mathcal{F}_{\underline d}$.

\begin{lemma}\label{startlemma}
There is a homomorphism of $\mathbb{Z}[q]$-algebras $$\phi:\mathbb{Z}[q]\Sigma(n,r)
/I(n,r)\rightarrow S_q(n, r)$$ defined by $$\phi(E_{i,\underline d})=e_{i,\underline d},
\phi(F_{i,\underline d})=f_{i,\underline d} \mbox{ and }
\phi(K_{\underline d})=k_{\underline d}.$$
\end{lemma}
\begin{proof}
By Lemma \ref{Lemma3.2}, the relations $P_{ij}$, $N_{ij}$
and $C_{ij}$ hold in $S_q(n,r)$, and so $\phi$ is an algebra homomorphism.
\end{proof}

We remark that the relations $P_{ij}$ and $N_{ij}$ hold in $S_q(n,r)$
also follows from Lemma \ref{Hallnumbers} 
and the proposition in Section 2 of \cite{RingelInv}, and
that the lemma can be deduced from Lemma 5.6 in \cite{BLM}.

The homomorphism $\phi$ not surjective in general, and so this
is not a presentation of the $q$-Schur algebra over $\mathbb{Z}[q]$, since 
for instance $[m]$ is not invertible in $\mathbb{Z}[q]$. 

\subsection{Change of rings}

We need the following change of rings lemma for presentations of
algebras using quivers with relations. The proof is similar to an argument
at the end of Chapter 5 in \cite{JZ}. Let $\psi:R\rightarrow S$ be a
homomorphism of commutative rings, which gives $S$ the structure
of an $R$-algebra. Let $\Sigma$ be a quiver, and let $I\subseteq R\Sigma$
be an ideal. There are induced map of $R$-algebras $\psi:R\Sigma
\rightarrow S\Sigma$ and $R\Sigma/I \rightarrow S\Sigma/S\psi(I)$,
where $S\psi(I)\subseteq S\Sigma$ is the ideal generated by $\psi(I)$.

\begin{lemma} \label{changerings}
The induced map $(R\Sigma/I)\otimes_{R}S\rightarrow S\Sigma/S\psi(I)$
is an isomorphism of $R$-algebras.
\end{lemma}
\begin{proof}
The natural isomorphism $R\otimes_R S\rightarrow S$ of $R$-algebras
induces an $R$-algebra isomorphism $$m:R\Sigma\otimes_R
S\rightarrow S\Sigma.$$ Applying the functor $-\otimes_RS$ to the short
exact sequence $$0 \rightarrow I \stackrel{i}{\rightarrow}R\Sigma \rightarrow
R\Sigma/I\rightarrow 0$$ gives us the exact sequence $$I\otimes_R
S \stackrel{j}{\rightarrow} S\Sigma \rightarrow (R\Sigma/I)\otimes_R
S\rightarrow 0$$ where $j=m\circ (i\otimes S)$, which shows that
$$(R\Sigma/I)\otimes_RS\simeq S\Sigma/\im(m\circ (i\otimes S)).$$
As $\im(m\circ (i\otimes S))=S\psi(I)$, the proof is complete.
\end{proof}

\subsection{Coefficients in $\mathbb{Q}(v)$}\label{injective}

Let $v=\sqrt{q}$ and $$S_v(n,r)= S_q(n, r)\otimes_{\mathbb{Z}[q]}\mathbb{Q}(v).$$

\begin{lemma} \label{qviso}
There is an isomorphism of $\mathbb{Q}(v)$-algebras $\mathbb{Q}(v)\Sigma(n,r)/\mathbb{Q}(v)I(n,r)
\rightarrow S_v(n,r),$ where $E_{i,\underline d}\mapsto e_{i, \underline d}$,
$F_{i, \underline d}\mapsto f_{i, \underline d}$ and $K_{\underline d}\mapsto k_{\underline d}$.
\end{lemma}
\begin{proof}
Let $\tilde{E}_i=\sum_{\underline d} v^{-d_i+1}e_{i,\underline d}$,
$\tilde{F}_i=\sum_{\underline d} v^{-d_{i+1}+1}f_{i,\underline d}$
and $\tilde{K}_{\underline d}=k_{\underline d}$, and by abuse of notation,
in this proof we let ${E}_i=\sum_{\underline d}e_{i,\underline d}$,
${F}_i=\sum_{\underline d} f_{i,\underline d}$
and $K_{\underline d}=k_{\underline d}$. Then both
$\{\tilde{E}_i, \tilde{F}_j, \tilde{K}_{\underline d}\}$ and
$\{{E}_i, {F}_j, {K}_{\underline d}\}$ generate $S_v(n, r)$. Moreover, by
a straightforward computation,
$\tilde{E}_i, \tilde{F}_j, \tilde{K}_{\underline d}$ satisfy
the defining relations in Theorem 4' in \cite{DG} if and only if
${E}_i, {F}_j, {K}_{\underline d}$ satisfy the relations $P_{ij}, N_{ij}$ and
$C_{ij}$. Therefore we have the isomorphism as required.
\end{proof}

\begin{proposition}\label{dgrelations}
The induced map $\phi\otimes{\mathbb{Q}(v)}:\mathbb{Z}
[q]\Sigma(n,r)/I(n,r)\otimes_{\mathbb{Z}[q]}\mathbb{Q}(v)
\rightarrow S_q(n, r)\otimes_{\mathbb{Z}[q]}\mathbb{Q}(v)$
is a $\mathbb{Q}(v)$-algebra isomorphism, where $\phi$ is as in Lemma
\ref{startlemma}.
\end{proposition}
\begin{proof}
By Lemma \ref{changerings}, the natural inclusion
$\psi:\mathbb{Z}[q]\rightarrow \mathbb{Q}(v)$, induces an isomorphism
$$\mathbb{Z}[q]\Sigma(n,r)/I(n,r)\otimes_{Z[q]}\mathbb{Q}(v)\simeq
\mathbb{Q}(v)\Sigma(n,r)/\mathbb{Q}(v)I(n,r),$$
which composed with the isomorphism in Lemma \ref{qviso} is 
$\phi\otimes \mathbb{Q}(v)$. Thus the proposition follows.
\end{proof}

Since $q$ is invertible in $S_v(n,r)$, we cannot obtain the $0$-Schur
algebra by specialising $q=0$. 

\subsection{A presentation of $q$-Schur algebra over $\mathcal{Q}$}

Now we choose an intermediate
ring $\mathbb{Z}[q]\subseteq \mathcal{Q} \subseteq
\mathbb{Q}(v)$ such that $q$ is non-invertible in $\mathcal{Q}$, and we will prove in
Section \ref{real} that $\phi\otimes \mathcal{Q}$ is an isomorphism.

Let $\mathcal{Q}$ be obtained from $\mathbb{Z}[q]$ by inverting 
all polynomials of the form $1+ qf(q)$. In particular, all 
$[m]$ for $m\in \mathbb{N}$ are inverted. Clearly, $\mathbb{Z}[q]\subseteq
\mathcal{Q} \subseteq \mathbb{Q}(q)$. Note that $q$ is not invertible
in $\mathcal{Q}$ and so the specialisation $q=0$ is possible.

\begin{proposition} \label{surjQ}
The induced map $\phi\otimes\mathcal{Q}:\mathbb{Z}
[q]\Sigma(n,r)/I(n,r)\otimes_{\mathbb{Z}[q]}\mathcal{Q}
\rightarrow S_q(n,r)\otimes_{\mathbb{Z}[q]}\mathcal{Q}$ is a surjective
$\mathcal{Q}$-algebra homomorphism.
\end{proposition}
\begin{proof}
The image of $\phi\otimes \mathcal{Q}$ is the
subalgebra of $S_q(n,r)\otimes_{\mathbb{Z}[q]}\mathcal{Q}$ generated
by the set of all $e_{i,\underline d}, f_{i,\underline d}$ and $k_{\underline d}$.
Lemma \ref{posmult} shows that the $\mathbb{Z}[q]$-subalgebra of
$S^+_q(n,r)$ generated by all $e_{i,\underline d}$ and $k_{\underline d}$
contains all $$[f, g]\prod [m_{ij}]!$$ where $g\subseteq f$ and $m_{ij}$
is the multiplicity of $M_{ij}$ as a direct summand in $f/g$. Since $[m]$ is
invertible in $\mathcal{Q}$ for any $m$, the image contains
$S^+_q(n,r)\otimes_{\mathbb{Z}[q]}\mathcal{Q}$. Similarly, the image
contains $S^-_q(n,r)\otimes_{\mathbb{Z}[q]}\mathcal{Q}$. By Theorem
\ref{2basis} the map is surjective. 
\end{proof}

By Lemma \ref{changerings}, $\mathcal{Q}\Sigma(n,r)/\mathcal{Q}I(n,r)\simeq \mathbb{Z}
[q]\Sigma(n,r)/I(n,r)\otimes_{Z[q]}\mathcal{Q}$. We have
the following theorem, which will be proven in Section \ref{real}. 

\begin{theorem} \label{laterproof}
The induced map $\phi\otimes\mathcal{Q}:\mathbb{Z}
[q]\Sigma(n,r)/I(n,r)\otimes_{\mathbb{Z}[q]}\mathcal{Q}
\rightarrow S_q(n,r)\otimes_{\mathbb{Z}[q]}\mathcal{Q}$ is 
a $\mathcal{Q}$-algebra isomorphism.
\end{theorem}

\section{A generic algebra}\label{genmult}

In this section let $k$ be algebraically closed.  We define a generic
multiplication of orbits in $\mathcal{F}\times \mathcal{F}$ and obtain
a $\mathbb{Z}$-algebra $G(n,r)$, which we call a generic algebra.
This multiplication generalises the one 
for positive $0$-Schur algebras in \cite{Su} and is
similar to the product defined by Reineke \cite{Reineke}
for Hall algebras. We also give generators for $G(n,r)$.

Let $\Delta:\mathcal{F}\times\mathcal{F}\times \mathcal{F} \rightarrow
(\mathcal{F}\times \mathcal{F})\times (\mathcal{F} \times \mathcal{F})$
be the morphism given by $$\Delta(p_1,p_2,p_3)=((p_1,p_2),(p_2,p_3)).$$
Let $$\pi:\mathcal{F}\times\mathcal{F}\times \mathcal{F} \rightarrow
\mathcal{F}\times\mathcal{F}$$ be the projection onto the left and right
component. The map $\pi$ is open, and $\Delta$ is a closed embedding.

Given two orbits $e_{A}$ and $e_{A'}$, define $$S(A,A')=\pi\Delta^{-1}
(e_{A}\times e_{A'})$$ That is, $S(A,A')$ is the union of the orbits
with non-zero coefficient in the product $e_{A}\cdot e_{A'}$ in $S_q(n,r)$.

\begin{lemma}\label{irred}
The closure of $S(A,A')$ in $\mathcal{F}\times\mathcal{F}$ is irreducible.
\end{lemma}
\begin{proof}
Let $[f_1,f_2]=e_{A}$, $[f_3,f_4]=e_{A'}$, and $S=\Delta^{-1}
(e_{A}\times e_{A'})$. We first show that  $S$ is irreducible. If
$f_2\not\simeq f_3$ then $S$ is empty, and we are done. So we
may assume that $f_2=f_3$. Let $(p_1,p_2,p_3)\in S$ then
there exists $g\in \gl(V)$ such that $(p_2,p_3)=g(f_3,f_4)$ and
$g(g^{-1}p_1,f_3,f_4)=(p_1,p_2,p_3)$, where $\gl(V)$ acts diagonally.
Since $(g^{-1}p_1,f_3)\simeq (f_1,f_3)$, there is an $a\in Autf_3$
such that $g^{-1}p_1=af_1$. Hence $S$ is the image of the morphism
$$Autf_3\times \gl(V) \rightarrow \mathcal{F}\times \mathcal{F}\times
\mathcal{F}$$ given by $$(a,g)\mapsto (gaf_1,gf_3,gf_4)$$
and is therefore irreducible. Now $S(A,A')=\pi(S)$, and so its closure is
irreducible.
\end{proof}

Since there are only finitely many orbits in $S(A,A')$, as a consequence
of Lemma \ref{irred}, we have the following corollary.

\begin{corollary} \label{uniqueorbit}
There is a unique open $\gl(V)$-orbit in $S(A,A')$.
\end{corollary}

We we define a new multiplication
$$e_{A}\star e_{A'}=e_{A''},$$ if $S(A,A')$ is non-empty and $e_{A''}$
is the open orbit in $S(A,A')$, and $$e_{A}\star e_{A'}=0$$ if $S(A,A')$
is empty.

\begin{proposition}
The free $\mathbb{Z}$-module $G(n,r)$ with the product
$\star$ is a $\mathbb{Z}$-algebra.
\end{proposition}
\begin{proof}
We need only to show that $\star$ is associative, that is, for any 
$[f_1, f_2], [f_3, f_4], [f_5, f_6]\in \mathcal{F}
\times \mathcal{F}/\gl(V)$, $$ ([f_1, f_2]\star[f_3, f_4])
\star [f_5, f_6] = [f_1, f_2]\star ([f_3, f_4] \star[f_5, f_6])$$

Following the definition, we see that if one side of the equality is zero, then
so is the other side. We now suppose that both sides are not zero, that is,
$f_2\simeq f_3$ and $f_4\simeq f_5$. By change of basis we may assume
that $f_2=f_3$ and $f_4=f_5$. Let \\

\begin{tabular} {ll}
& $T_1=\{(p_1, p_2, p_3, p_4)\mid (p_1, p_2)\simeq (f_1, f_2), (p_2, p_3)
\simeq (f_3, f_4), (p_3, p_4)\simeq (f_5, f_6)\},$ \\ &  $T_2=
\{(p_1, p_3, p_4)\mid \exists \; p \mbox{ such that } (p_1, p)\simeq (f_1, f_2),
(p, p_3)\simeq (f_3, f_4), (p_3, p_4)\simeq (f_5, f_6)\},$ \\ &
$T_3=\{(p_1, p_2, p_4)\mid  \exists \; p \mbox{ such that }(p_1, p_2)\simeq
(f_1, f_2), (p_2, p)\simeq (f_3, f_4), (p, p_4)\simeq (f_5, f_6)\}, $ \\ &
$T_4=\{(p_1, p_4)\mid \exists \; p, p' \mbox{ such that }(p_1, p)\simeq (f_1,
f_2), (p, p')\simeq (f_3, f_4), (p', p_4)\simeq (f_5, f_6)\}.
$\end{tabular} \\

We have natural surjections $$\pi_{ij}: T_i\ra T_j$$ for $(i, j)=(1, 2), (1, 3),
(2, 4), (3, 4)$. Similar to the proof of Lemma \ref{irred}, we see that
$T_1$ is irreducible, and so the closures of all the $T_i$ are irreducible.
In particular, there is a unique open orbit $\mathcal{O}$ in $T_4$. Then
$\pi_{24}^{-1}(\mathcal{O})$ intersects with the open subset of $T_2$,
consisting of triples $(p_1, p_3, p_4)$ with $[p_1, p_3]$ open in $S(A,A')$.
That is, $([f_1, f_2]\star[f_3, f_4])\star [f_5, f_6] $ is the open orbit
$\mathcal{O}$ in $T_4$. Similarly, $[f_1, f_2] \star ([f_3, f_4] \star[f_5, f_6])$
is also the open orbit  $\mathcal{O}$. Therefore the equality holds and so
$\star$ is associative.
\end{proof}

The following is a direct consequence of the definition of the product in
$G(n,r)$.

\begin{corollary}
The set $\mathcal{F}\times \mathcal{F}/\gl(V)$ is a multiplicative basis of
$G(n,r)$.
\end{corollary}

In addition to the basis of $G(n,r)$ consisting of orbits $[f_1,f_2]$ we can 
also consider bases analogous to the bases $\mathcal{B}$ and $\mathcal{B}'$ 
defined in Section 4 for the $q$-Schur algebra.
We conclude this section by showing that these three bases of $G(n,r)$
coincide.

\begin{lemma}\label{genericmult}
Let $(f_1, f_2)\in \mathcal{F}\times \mathcal{F}$. Then $[f_1, f_1+f_2]
\star [f_1+f_2, f_2]=[f_1, f_2]=[f_1,f_1\cap f_2]\star [f_1\cap f_2,f_2]$.
\end{lemma}
\begin{proof}
We prove the first equality. Let $e_{A}=
[f_1, f_1+f_2]$ and $e_{A'}=[f_1+f_2, f_2]$. We prove that the orbit
$[f_1,f_2]$ is open in $S(A,A')$. For any $(f'_1,f'_2)\in S(A,A')$, 
$f'_1+f'_2$ is isomorphic to a subflag of $f_1+f_2$. By Lemma \ref{bases1}, for
$(f '_1, f '_2) \in S(A, A')$, we have $(f '_1, f '_2)\simeq (f_1, f_2)$
if and only if $f '_1+f '\simeq f_1+f_2$. That the dimension of
$f'_1+f'_2$ is maximal is an open condition, and therefore
$e_{A}\star e_{A'}=[f_1,f_2]$.

Similarly,  $[f_1, f_1\cap f_2] \star [f_1\cap f_2, f_2]=[f_1, f_2]$.
\end{proof}

We now prove that the $\mathbb{Z}$-algebra
$G(n,r)$ is generated by the orbits $e_{i,\underline d}$, $f_{i,\underline d}$
and $k_{\underline d}$.
Recall that a representation $X$ is said to be a generic extension of $N$
by $M$, if the stabiliser of $X$ is minimal among all representations that
are extensions of $N$ by $M$. 

\begin{lemma} \cite{Su} \label{sulemma}
Let $f\supseteq g\supseteq h$ be flags. Then $[f,h]=[f,g]\star[g,h]$
if and only if $f/h$ is a generic extension of $f/g$ by $g/h$.
\end{lemma}

For an interval $[i,j]$ in $\{1,\cdots,n\}$ and $\underline d\in D(n,r)$
with $\underline d-\alpha_{j+1}$ non-negative, let 
$$e(i,j,\underline d)=e_{i,\underline d+
\alpha_{i+1}-\alpha_{j+1}}\star\cdots \star e_{j,\underline d}.$$
Similaly, let $f(i,j,\underline d)= f_{j,\underline d-
\alpha_{i+1}+\alpha_{j+1}}\star \cdots 
\star f_{i,\underline d}$ for $\underline d-\alpha_{i+1}$ non-negative.

\begin{lemma} \label{genindec}
Let $f\supseteq h$ be flags with $h\in \mathcal{F}_{\underline d}$ and
$f/h\simeq M_{ij}$. Then $[f,h]=e(i,j,\underline d)$ and $[h,f]=f(i,j,\underline d+\alpha_{i}-\alpha_{j+1})$.
\end{lemma}
\begin{proof}
If $i=j$, then $[f,h]=e_{i,\underline d}$. Now assume $j>i$.
Then there is $f\supseteq g\supseteq h$ with
$f/g\simeq M_{i,j-1}$ and $g/h\simeq M_{jj}$. Since $f/h$ is a generic
extension of $f/g$ by the simple $g/h$, the
lemma follows from Lemma \ref{sulemma} by induction.
\end{proof}

Using the order $\leq$ on representations defined in Section 3, we can write
each orbit $[f,g]$ with $f\supseteq g$ as a product over indecomposable 
summands of $f/g$.

\begin{lemma} \label{genallmod} Let $f\supseteq g$  be flags with $f/g\simeq 
\bigoplus_{i=1}^t M_i$ and $M_{i}\leq M_{i+1}$. Then there is a
filtration $f=f_t \supset f_{t-1} \supset \cdots \supset f_0 =
g \supset 0$ with indecomposable factors $M_i=f_i/f_{i-1}$
and $[f,g]=[f_t,f_{t-1}]\star\cdots\star[f_1,f_0]$.
\end{lemma}
\begin{proof}
The lemma follows from the vanishing of extension groups along the filtration and Lemma
\ref{sulemma}.
\end{proof}

\begin{lemma} \label{generatorsGNR}
The $\mathbb{Z}$-algebra $G(n,r)$ is generated by the orbits $e_{i,\underline d}$,
$f_{i,\underline d}$ and $k_{\underline d}$.
\end{lemma}
\begin{proof}
Lemma \ref{genindec} and Lemma \ref{genallmod} imply that 
any orbit $[f,g]$ with
$f\supseteq g$ is in the subalgebra of $G(n,r)$ generated by $e_{i,\underline d}$ and $k_{\underline d}$. The lemma now follows from Lemma \ref{genericmult}.

\end{proof}

Following Lemma \ref{genericmult}, \ref{genindec} and \ref{genallmod}, we obtain
the following basis of $G(n,r)$ in terms 
the generators $e_{i,\underline d}$ and $f_{i,\underline d}$.

\begin{lemma} \label{monbasis}
The $\mathbb{Z}$-algebra $G(n,r)$ has a basis consisting of all $k_{\underline d}$ and all non-zero monomials
$$e(i_s,j_s,\underline d_s)\star \cdots \star e(i_1,j_1,\underline d_1)
\star f(i'_1,j'_1,\underline d'_1) \star \cdots \star f(i'_t,j'_t,\underline d'_t),$$
where $M_{i_lj_l}\leq M_{i_{l+1}j_{l+1}}$, $M_{i'_lj'_l}\leq M_{i'_{l+1}j'_{l+1}}$
and $\underline  d_1\geq \sum_{l}\alpha_{j_l+1}+\sum_l\alpha_{j'_l+1}.$ 
\end{lemma}
\begin{proof}
We need only show that for any such monomial $$e(i_s,j_s,\underline d_s)\star \cdots \star e(i_1,j_1,\underline d_1)
\star f(i'_1,j'_1,\underline d'_1) \star \cdots \star f(i'_t,j'_t,\underline d'_t)$$
there is an orbit $[f,g]$ such that 
$$[f,f\cap g]=e(i_s,j_s,\underline d_s)\star \cdots \star e(i_1,j_1,\underline d_1)
\mbox{ and } [f\cap g,g]= f(i'_1,j'_1,\underline d'_1)\star \cdots \star f(i'_t,j'_t,\underline d'_t).$$
Write $\underline d = \underline c + \underline d' + \underline d''$, where $\underline d'=\sum_{l}\alpha_{j_l+1}$ and $\underline d''=\sum_l\alpha_{j'_l+1}$. Consider
$P(d)$ as a flag in $V$, and decompose as
$P(d)=P(c)\oplus P(d')\oplus P(d'')$. Let $Q(d')$ and $Q(d'')$ be flags containing
$P(d')$ and $P(d'')$, respectively, such that 
$$Q(d')/P(d')\simeq \bigoplus M_{i_lj_l} \mbox{ and  } Q(d'')/P(d'')\simeq 
\bigoplus M_{i'_lj'_l}.$$  Let $f=P(c)\oplus Q(d')\oplus P(d'')$ and $g=P(c)\oplus P(d')\oplus Q(d'')$, then $[f,g]$ is an orbit as required.
\end{proof}

We compute the multiplication in $G(n,r)$ of an arbitrary element with
a generator.

\begin{lemma} \label{Lemma3.20Schur}
Let $e_A\subseteq \mathcal{F}_{\underline d} \times \mathcal{F}$.
\begin{itemize}
\item[i)] If $d_{i+1}>0$, then $e_{i,d}\star e_A=e_X$ where
$X=A+E_{i,p}-E_{i+1,p}$ and $p=max\{j\mid A_{i+1,j}>0\}$.
\item[ii)] If $d_{i}>0$, then $f_{i,d}\star e_A=e_Y$ where
$Y=A-E_{i,p}+E_{i+1,p}$ and $p=min\{j\mid A_{i,j}>0\}$.
\end{itemize}
\end{lemma}
\begin{proof}
We prove i). By Lemma \ref{Lemma3.2}, the orbit $e_X$ has a non-zero
coefficient in the product $e_{i,\underline d}\cdot e_A$ in $S_q(n,r)$. Now, by
Lemma 2.2 in \cite{BLM}, among all terms $A+E_{i,j}-E_{i+1,j}$ with
$A_{i+1,j}>0$, the elements in the orbit $e_X$ has the smallest stabiliser, and
so $e_{i,d}\star e_A=e_X$.

The proof of ii) is similar.
\end{proof}

\section{A geometric realisation of the $0$-Schur algebra} \label{real}

In this section we first give a presentation of $G(n,r)$ using quivers and
relations. Then we show that $S_0(n,r)$ and $G(n,r)$ are
isomorphic as $\mathbb{Z}$-algebras by an isomorphism which is the identity
on the closed orbits $e_{i,\underline d}$, $f_{i,\underline d}$ and
$k_{\underline d}$. Finally, we prove Theorem \ref{laterproof}.

\subsection{A presentation of $G(n,r)$}

Let $\Sigma(n,r), E_i$ and $F_i$ be as in Section 5. 
Let $$P_{ij}(0)=\left\{  \begin{matrix}E_i^2E_j-E_iE_jE_i
\mbox{ for } i=j-1, \\ -E_iE_jE_i+ E_jE^2_{i} \mbox{ for }
i=j+1, \\ E_{i}E_{j}-E_{j}E_{i}, \mbox{ otherwise;}\end{matrix}\right.$$
$$N_{ij}(0)=\left\{  \begin{matrix} -F_iF_jF_i+ F_jF^2_{i} \mbox{ for }
i=j-1, \\ F_i^2F_j-F_iF_jF_i \mbox{ for }  i=j+1, \\ F_{i}F_{j}-F_{j} F_{i},
\mbox{ otherwise; } \end{matrix}\right. $$ and $$C_{ij}(0)=E_{i}F_{j}-
F_{i}E_{j}-\delta_{ij} \sum_{\underline d} \lambda_{ij}(\underline d)
\cdot K_{\underline d},$$ where $$\lambda_{ij}(\underline d)=\left\{
\begin{matrix} 1 \mbox{ if } d_i>d_{i+1}=0,\\ -1  \mbox{  if }
 d_{i+1}>d_i=0, \\ 0 \mbox{ otherwise. } \end{matrix}\right.$$
That is, $P_{ij}(0), N_{ij}(0)$ and $C_{ij}(0)$
are obtained by evaluating $P_{ij}, N_{ij}$ and $C_{ij}$ at $q=0$.

Let $I_0(n,r)\subseteq \mathbb{Z}\Sigma(n,r)$ be the
ideal generated by $P_{ij,d}(0)$, $N_{ij,d}(0)$, and $C_{ij,d}(0)$.

\begin{lemma} 
$\mathbb{Z}\Sigma(n,r)/I_0(n,r)$ has a multiplicative basis of paths in $\Sigma(n,r)$.
\end{lemma}
\begin{proof}
The lemma holds since each relation $P_{ij,\underline d}(0)$,
$N_{ij,\underline d}(0)$, and  $C_{ij,\underline d}(0)$ is a 
binomial in $E_{i,\underline d}$, $F_{i,\underline d}$ and $K_{\underline d}$.
This is obvious for $P_{ij,\underline d}(0)$, $N_{ij,\underline d}(0)$. For  $C_{ij,\underline d}(0)$, if the coefficient of 
$K_{\underline d}$ is nonzero then either $C_{ij,\underline d}(0)=K_{\underline d}
E_{i}F_{j}K_{\underline d} - K_{\underline d}$ or $C_{ij,\underline d}(0)=
K_{\underline d} F_{i}E_{j}K_{\underline d}-K_{\underline d}$. 
\end{proof}

For an interval $[i,j]$ in $\{1,\cdots,n\}$ and $\underline d\in D(n,r)$
with $\underline d-\alpha_{j+1}$ non-negative, let 
$$E(i,j,\underline d)=E_{i,\underline d+
\alpha_{i+1}-\alpha_{j+1}}\cdots E_{j,\underline d}$$ 
and $F(i,j,\underline d)=F_{j,\underline d-\alpha_{i+1}+\alpha_{j+1}}\cdots F_{i,\underline d}$ for $\underline d-\alpha_{i+1}$ non-negative.

\begin{theorem} \label{presgnr}
The map $\eta:\mathbb{Z}\Sigma(n,r)/I_0(n,r)\rightarrow G(n,r)$ given by
$\eta(E_{i,\underline d})=e_{i,\underline d}$,  $\eta(F_{i,\underline d})=
f_{i,\underline d}$ and $\eta(K_{\underline d})=k_{\underline d}$ is an 
isomorphism of $\mathbb{Z}$-algebras.
\end{theorem}
\begin{proof}
By Lemma \ref{Lemma3.20Schur}, it is straightforward to check that $e_{i,\underline d}$, 
$f_{i,\underline d}$, and $k_{\underline d}$ satisfy the relations 
$P_{ij,d}(0)$, $N_{ij,d}(0)$, and $C_{ij,d}(0)$. Thus $\eta$ is well-defined.
Also, Lemma \ref{generatorsGNR} implies that the map is surjective.  It remains
to prove that $\eta$ is injective.

We claim that,  modulo the relations in $I_0(n,r)$, any path 
$p$ in $\Sigma(n,r)$ is either equal to $k_{\underline d}$ or a path of the form
$$E(i_s,j_s,\underline d_s) \cdots E(i_1,j_1,\underline d_1)
F(i'_1,j'_1,\underline d'_1)\cdots F(i'_t,j'_t,\underline d'_t),$$
satisfying the conditions in Lemma \ref{monbasis}. Note that
such a path is mapped onto one of monomial basis elements of Lemma \ref{monbasis},
and so $\eta$ is injective. 

We prove the claim by induction on the length of $p$. 
If $p$ has length less than or equal to one, it is equal to $k_{\underline d}$
or one of the arrows $F_{i,\underline d}$ and $E_{i,\underline d}$, and so the
claim follows.
Assume that $p$ has length greater than one. Then we have 
$$p=p'F_{i,\underline c} \mbox{ or } p=p'E_{i,\underline c}$$ where $p'$ is a non-trivial
path of smaller length, and so by induction has the required form $$p'=EF=E(i_s,j_s,\underline d_s) \cdots E(i_1,j_1,\underline d_1)
 F(i'_1,j'_1,\underline d'_1)\cdots F(i'_t,j'_t,\underline d'_t),$$ where $E$ and $F$ are
products of $E_{j,\underline d}$ and $F_{j,\underline d}$, respectively.

We first consider $p=p'E_{i,\underline c}$. If $p'$ contains no $F_{j,\underline d}$,
then the claim follows using the relations $P_{ab,\underline d}(0)$. Otherwise,
by the relations
$C_{ab,\underline d}(0)$, either 
$p=EF'$ with the length of $F'$ smaller than that of $F$ or 
$p=E E_{i,\underline d_1-\alpha_i+\alpha_{i+1}}F'$
with each factor $F(i'_l,j'_l,\underline d'_l)$ in $F$ 
replaced with a factor $F(i'_l,j'_l,\underline c'_l)$.
In the first case, the claim follows by induction.  Otherwise, by 
the relations $P_{ab,\underline d}(0)$, there are two possibilities. First, there exists a minimal $m$ with $j_m=i-1$. Then 
$EE_{i,\underline d_1-\alpha_i+\alpha_{i+1}}F'$ is equal to $$E(i_s,j_s,d_s)\cdots E(i_{m-1},j_{m-1},\underline d_{m-1})E(i_m,j_m+1,\underline c_{m})E(i_{m+1},j_{m+1},\underline c_{m+1})
\cdots E(i_1,j_1,\underline c_1)F'.$$
We have $$c_1=d_1-\alpha_i+\alpha_{i+1}\geq \sum_{l}\alpha_{j_l+1} -\alpha_{i} + \alpha_{i+1} + \sum_l\alpha_{j'_l+1}=\sum_{l\neq m}\alpha_{j_l+1} + \alpha_{j_m+1} + \sum_l\alpha_{j'_l+1}.$$ Moreover, 
again using the relations $P_{ab,\underline d}(0)$, the factors
can be reordered  (up to change of $d_l,c_m$) to obtain a path of the required form. 

Second, there is no such $m$ with $j_m=i-1$. Then 
$$EE_{i,\underline d_1-\alpha_i+\alpha_{i+1}}F'=E(i_s,j_s,\underline d_s)\cdots E(i_{m},j_{m},\underline d_m)E(i,i,\underline c_m)E(i_{m-1},j_{m-1},
\underline c_{m-1})
\cdots E(i_1,j_1,\underline c_1)F',$$ with $j_{m-1}\leq i$ and $j_{m}>i$. 
In order to show that this path is of the required form, we need only to prove the 
inequality 
$$c_1=d_1-\alpha_i+\alpha_{i+1}\geq \sum_l\alpha_{j_l+1}+\alpha_{i+1}+
\sum_l\alpha_{j'_l+1}.$$ 
Clearly, the inequality holds for each component different
from $i$.  Since there are no $m$ with $j_m=i-1$, the sum $\sum_l\alpha_{j_l+1}$ 
contain no $\alpha_i$. Since $FE_{i,c}=E_{i,c_1}F'$ with the length of $F'$ equal 
to that of $F$, we must have $(d_1-\alpha_i)_i \geq (\sum_l\alpha_{j'_l+1})_i$ 
and so the inequality follows.

Finally, we consider $p=p'F_{i,\underline c}$, where $p'$ is a path
of the required form $p'=EF$ as above. If there are no factor 
$E_{j,\underline d}$ in $p'$, then
the claim follows from the relations $N_{ab,\underline d}(0)$. 
Otherwise $p=E'E_{j,\underline d_1}F$, which following $C_{ab,\underline d}(0)$
is either $p=E'F'$ with $F'$ shorter than $F$, or $p=E'F'E_{j,\underline c}$
which is then the case showed above. So the claim holds.
\end{proof}

\subsection{A geometric realisation of $S_0(n,r)$}

We now prove the main result of this section.

\begin{theorem} \label{geomreal}
The map $$\psi:G(n,r)\rightarrow S_0(n,r)$$ defined by $\psi(e_{i,\underline d})
=e_{i,\underline d}$, $\psi(f_{i,\underline d})=f_{i,\underline d}$ and
$\psi(k_{\underline d})=k_{\underline d}$ is an isomorphism of
$\mathbb{Z}$-algebras.
\end{theorem}
\begin{proof}
From Proposition \ref{surjQ},
we have the surjective $\mathcal{Q}$-algebra homomorphism
$$\mathcal{Q}\Sigma(n,r)/\mathcal{Q}I(n,r))
\rightarrow S_q(n,r)\otimes_{\mathbb{Z}[q]}\mathcal{Q},$$
which, since $\mathcal{Q}/q\mathcal{Q}\simeq \mathbb{Z}$, 
induces a surjective $\mathbb{Z}$-algebra homomorphism
$$\mathcal{Q}\Sigma(n,r)/\mathcal{Q}I(n,r))\otimes_{\mathcal{Q}}
\mathcal{Q}/q\mathcal{Q}\rightarrow 
S_q(n,r)\otimes_{\mathbb{Z}[q]}\mathcal{Q}\otimes_{\mathcal{Q}}
\mathcal{Q}/q\mathcal{Q}.$$

Following the definition of $S_0(n,r)$ and
the isomorphisms $\mathcal{Q}/q\mathcal{Q}
\simeq \mathbb{Z}[q]/q\mathbb{Z}[q]\simeq \mathbb{Z}$, we have
$$S_q(n,r)\otimes_{\mathbb{Z}[q]}\mathcal{Q}\otimes_{\mathcal{Q}} \mathcal{Q}/q\mathcal{Q}=S_0(n,r)$$  
and by Lemma \ref{changerings}
$$(\mathcal{Q}\Sigma(n,r)/\mathcal{Q}I(n,r))
\otimes_{\mathcal{Q}}\mathcal{Q}/q\mathcal{Q}
\simeq (\mathbb{Z}\Sigma(n,r)/I_0(n,r)).$$ 
So there is a surjective $\mathbb{Z}$-algebra homomorphism 
$\mathbb{Z}\Sigma(n,r)/I_0(n,r)) \rightarrow S_0(n,r)$ given by
$E_{i,\underline d}\mapsto e_{i,\underline d}$, 
$F_{i,\underline d}\mapsto f_{i,\underline d}$, 
$K_{i,\underline d}\mapsto k_{i,\underline d}$. The theorem now follows
from Theorem \ref{presgnr}, since $G(n,r)=S_0(n,r)$ as $\mathbb{Z}$-modules.
\end{proof}

\begin{corollary} \label{multbasisG}
The set $\psi(\mathcal{F}\times \mathcal{F}/\gl(V))$ is a
multiplicative basis for $S_0(n,r)$.
\end{corollary}

\subsection{Proof of Theorem \ref{laterproof}}

By Proposition \ref{surjQ}, 
the map $\phi\otimes \mathcal{Q}$ induces a short exact sequence 
$$0 \rightarrow K \rightarrow \mathcal{Q}\Sigma(n,r)/\mathcal{Q}I(n,r)\rightarrow S_q(n,r)\otimes_{\mathbb{Z}[q]}\mathcal{Q}\rightarrow 0.$$ Since $S_q(n,r)
\otimes_{\mathbb{Z}[q]}\mathcal{Q}$ is
a free $\mathcal{Q}$-module, applying $-\otimes_{\mathcal{Q}}\mathcal{Q}/q\mathcal{Q}$ gives the exact
sequence $$0 \rightarrow K\otimes_{\mathcal{Q}}\mathcal{Q}/q\mathcal{Q} \rightarrow \mathcal{Q}\Sigma(n,r)/\mathcal{Q}I(n,r)\otimes_{\mathcal{Q}}\mathcal{Q}/q\mathcal{Q}\rightarrow S_q(n,r)\otimes_{\mathbb{Z}[q]}\mathcal{Q}\otimes_{\mathcal{Q}}\mathcal{Q}/q\mathcal{Q}\rightarrow 0.$$

As in the proof of Theorem \ref{geomreal}, we have isomorphisms
$S_q(n,r)\otimes_{\mathbb{Z}[q]}\mathcal{Q}\otimes_{\mathcal{Q}} \mathcal{Q}/q\mathcal{Q}=S_0(n,r)$  and
$(\mathcal{Q}\Sigma(n,r)/\mathcal{Q}I(n,r))\otimes_{\mathcal{Q}}\mathcal{Q}/q\mathcal{Q}
\simeq (\mathbb{Z}\Sigma(n,r)/I_0(n,r)).$
Furthermore, via these two isomorphisms the map $\phi \otimes \mathcal{Q} \otimes \mathcal{Q}/q\mathcal{Q}$ is the composition of the isomorphism $\mathbb{Z}\Sigma(n,r)/I_0(n,r)\simeq G(n,r)$ in
Theorem \ref{presgnr} and the isomorphism  $G(n,r)\simeq S_0(n,r)$ in 
Theorem \ref{geomreal}. 
Therefore $K\otimes_{\mathcal{Q}}\mathcal{Q}/q\mathcal{Q}=K/qK=0$. Now
by Nakayama's lemma there is an element $r=1+qf(q)\in \mathcal{Q}$ such that
$rK=0$. Since $r$ is invertible in $\mathcal{Q}$, we have $K=0$. Thus $\phi\otimes \mathcal{Q}$ is an isomorphism.

\section{The degeneration order on pairs of flags}

In this section let $k$ be algebraically closed. We describe the
degeneration order on $\gl(V)$-orbits in $\mathcal{F}\times
\mathcal{F}$ using quivers and the symmetric group $S_r$.

Let $\Gamma=\Gamma(n)$ be the quiver of type $\mathbb{A}_{2n-1}$,
$$\Gamma: \xymatrix{1_L \ar[r] & 2_L \ar[r] & \cdots \ar[r] & n & \ar[l]
\cdots & \ar[l] 2_R & 1_R \ar[l]}$$  constructed by joining two linear
quivers $\Lambda_L=\Lambda_L(n)$ and $\Lambda_R=\Lambda_R(n)$
at the vertex $n$. Often it will be clear from the context which side of
$\Gamma$ we are considering, and then we drop the subscripts on the
vertices.

A pair $(f,f')\in \mathcal{F}\times \mathcal{F}$ is a representation of
$\Gamma$, where $f$ is supported on $\Lambda_L$, $f'$ is supported on
$\Lambda_R$. Conversely, any representation $M$ of $\Gamma$ with
$\mathrm{dim} M_n=r$, which is projective when restricted to both
$\Lambda_L$ and $\Lambda_R$ determines uniquely an orbit of pair of
flags $[f, f']\in \mathcal{F}\times \mathcal{F}/\gl(V)$. Moreover,
two pairs of flags are isomorphic if and only if the corresponding
representations are isomorphic.

For integers $i,j\in \{1,\cdots,n\}$, let $N_{ij}$ be the indecomposable
representation of $\Gamma$ which is equal to the indecomposable
projective representations $M_{in}$ and $M_{jn}$ when restricted to
$\Lambda_L$ and $\Lambda_R$, respectively. A representation $N$ of
$\Gamma$ which is projective when restricted to $\Lambda_L$ and
$\Lambda_R$, and $\dim N_n=r$, decomposes up to isomorphism as
$$\label{eq1}N \simeq \bigoplus^r_{l=1} N_{i_lj_l}.$$ In this section
we always assume that $j_1\leq j_2 \leq \cdots \leq j_r$. The variety
of representations that has such a decomposition is an open subset of
the variety of representations of $\Gamma$ with dimension vector 
the same as $N$. We shall view representations
in this subvariety as pairs of flags in $V$.

Let $\leq_{deg}$ denote the degeneration order on isomorphism classes
of representations of $\Gamma$. That is, $M\leq_{deg}N$ for two
representations $M$ and $N$, if $N$ is contained in the
closure of the orbit of $M$ in the space of all representations.
The degeneration order on pairs of flags is also denoted by
$\leq_{deg}$, since there is a degeneration between two pair of flags
if and only if there is a degeneration between the corresponding
representations of $\Gamma$. 

Since $\Gamma$ is a Dynkin quiver,
by a result of Bongartz \cite{bongartz}, the degeneration $\leq_{deg}$
is the same as the degeneration $\leq_{ext}$ given by a sequence of 
extensions. That is, if there  is an extension
$$\xymatrix{0 \ar[r] & N' \ar[r] & M \ar[r] & N'' \ar[r] & 0 },$$
then $M\leq_{ext}N'\oplus N''$, and more generally $\leq_{ext}$ is the
transitive closure.

The symmetric group $S_r$ of permutations of the set $\{1,\cdots,r\}$ acts
on representations with a decomposition $N=\bigoplus^r_{l=1} N_{i_lj_l}$
by $$\sigma N=\bigoplus^r_{l=1}N_{i_{\sigma l}j_l}$$ for $\sigma\in S_r$.

The following facts are the key lemmas on degenerations in
$\mathcal{F}\times \mathcal{F}$. For the sake of completeness we include
a brief sketch of the proofs.

\begin{lemma} \label{deglemma}
Let $N=\bigoplus^r_{l=1}N_{i_lj_l}$ be a decomposition as above, and let $(t,s)$
with $t<s$ be a transposition. Then  $N<_{deg}(t,s)N$ if and only if
$i_t>i_s$.
\end{lemma}
\begin{proof}
Assume that $i_t>i_s$.
There is a short exact sequence $$\xymatrix{0 \ar[r] & N_{i_tj_s} \ar[r] & 
N_{i_tj_t} \oplus N_{i_sj_s} \ar[r] & N_{i_sj_t} \ar[r] & 0 }.$$ Since
every extension degenerates to the trivial extension we have
$$N_{i_tj_t} \oplus N_{i_sj_s}\leq_{deg} N_{i_tj_s}\oplus N_{i_sj_t}$$
and therefore $N<_{deg}(t,s)N$.

Conversely, assume that $i_t\leq i_s$. By comparing the dimensions of the
stabilisers of $N$ and $(t,s)N$ we see that $N\not<_{deg}(t,s)N$.
\end{proof}

We say that a degeneration $M\leq_{deg}N$ is minimal if $M\not\simeq N$ and 
$M\leq_{deg} X \leq_{deg} N$ implies $X\simeq M$ or $X\simeq N$.

\begin{lemma} \label{deglemma2}
Let $N=\bigoplus^r_{l=1}N_{i_lj_l}$ and $M\leq_{deg}N$ be minimal. Then there exists
a transposition $(t,s)$ such that $M\simeq (t,s)N$. 
\end{lemma}
\begin{proof}
Since $M\leq_{deg}N$ is minimal, there is a non-split extension 
$$\xymatrix{0 \ar[r] & N' \ar[r] & M \ar[r] & N'' \ar[r] & 0, }$$
where $N\simeq N'\oplus N''$. We may choose summands $N_{i_sj_s}$ and $N_{i_tj_t}$ 
of  $N'$ and $N''$, respectively, such that taking pushout along the projection $N'\ra N_{i_sj_s}$
and then pullback along the inclusion $N_{i_tj_t}\ra N''$ gives us a non-split extension
$$\xymatrix{0 \ar[r] & N_{i_sj_s} \ar[r] & M' \ar[r] & N_{i_tj_t} \ar[r] & 0. }$$
This extension is of the form of the extension in the proof of Lemma \ref{deglemma}.
Hence $$M'<_{deg} N_{i_sj_s} \oplus N_{i_tj_t},$$ and so 
$$M'\oplus (N'/N_{i_sj_s}) \oplus (N''/N_{i_tj_t})<_{deg} (N'/N_{i_sj_s}) \oplus (N''/N_{i_tj_t})\oplus N_{i_sj_s} \oplus N_{i_tj_t}\simeq N.$$
By the construction of $M'$, $$M\leq_{deg} M'\oplus (N'/N_{i_sj_s}) \oplus (N''/N_{i_tj_t}),$$
so by the minimality of the degeneration, $$M\simeq M'\oplus (N'/N_{i_sj_s}) \oplus (N''/N_{i_tj_t}),$$
and so the lemma follows.
\end{proof}

There is a unique closed orbit in
$\mathcal{F}_{\underline d}\times \mathcal{F}_{\underline e}$. We
describe a corresponding representation.

\begin{lemma} \label{closedorbit}
The orbit of a pair of flags corresponding to a representation $N$ is
closed, if and only if $N\simeq \bigoplus^r_{l=1} N_{i_lj_l}$ with
$i_l \leq i_{l+1}$ for all $l=1,\cdots,r-1$.
\end{lemma}
\begin{proof}
A representation $N=\bigoplus^r_{l=1} N_{i_lj_l}$ with
$ i_l\leq i_{l+1}$ does not have
degenerations according to Lemma \ref{deglemma}. Hence $N$,
and therefore also the corresponding pair of flags, has a closed orbit.

Conversely, if $i_l>i_{l+1}$, then $N$ has a degeneration again by 
Lemma \ref{deglemma}, and so the orbit of $N$ is not closed.
\end{proof}

Alternatively, we may prove the lemma by observing that among
all representations of the form $N=\bigoplus^r_{l=1} N_{i_lj_l}$
the representation with $i_l\leq i_{l+1}$
has a stabiliser of maximal dimension, and so this representation
has a closed orbit. The stabiliser in this case is a parabolic in $\gl(V)$.

There is a unique open orbit in $\mathcal{F}_{\underline d}\times
\mathcal{F}_{\underline e}$ with a corresponding representation
given as follows. The proof is similar to the proof of the previous
lemma.

\begin{lemma} \label{openorbit}
The orbit of a pair of flags corresponding to a representation $N$ is
open, if and only if $N\simeq \bigoplus^r_{l=1} N_{i_lj_l}$ with
$i_l\geq i_{l+1}$ for all $l=1,\cdots,r-1$.
\end{lemma}

Similar to the closed orbit, a representation of the form $N\simeq
\bigoplus^r_{l=1} N_{i_lj_l}$ with $i_l\geq i_{l+1}$
has a stabiliser of minimal dimension, and so the orbit is
open. The stabiliser in this case is the intersection of two opposite
parabolics in $\gl(V)$. Such stabilisers are called biparabolic
or seaweeds \cite{derg}. The stabiliser of an arbitrary pair of flags
is equal to the intersection of two parabolics in $\gl(V)$.

Let $o_{\underline d,\underline e}$ denote the unique open orbit and
$k_{\underline d,\underline e}$ the unique closed orbit in
$\mathcal{F}_{\underline d}\times \mathcal{F}_{\underline e}$. Then
$k_{\underline d}=k_{\underline d,\underline d}$ and we let
$o_{\underline d}=o_{\underline d,\underline d}$. Let $\tau\in S_r$,
then $\tau o_{\underline d,\underline e}\in \mathcal{F}\times
\mathcal{F}/\gl(V)$ denotes the orbit of pairs of flag corresponding
to the representation $\tau N$, where $N=\bigoplus^r_{l=1} N_{i_lj_l}$
with $i_{l+1}\leq i_l$, $\underline d=i_1+\cdots+
i_n$ and $\underline e=j_1+\cdots+j_n$. Similarly, we let $\tau
k_{\underline d,\underline e}$ be defined as the orbit corresponding to $\tau N$,
where $N=\bigoplus^r_{l=1} N_{i_lj_l}$ with $i_{l+1}\geq i_l$.

\section{Idempotents from open orbits} \label{open}

Let $M(n,r)$ be the submodule of $G(n,r)$ with basis the
open orbits in $\mathcal{F}\times \mathcal{F}$. In this section
we prove that $M(n,r)$ is a subalgebra $G(n,r)$ which is also a
direct factor. We also show that $M(n,r)$ is isomorphic to
the $\mathbb{Z}$-algebra of $|D(n,r)|\times |D(n,r)|$-matrices,
where $|D(n,r)|$ is the number of decompositions of $r$
into $n$ parts.

We start with two lemmas relating degeneration and multiplication
in $G(n,r)$. Let $\leq_{deg}$ be the degeneration
order on orbits in $(\mathcal{F}\times \mathcal{F})\times
(\mathcal{F}\times \mathcal{F})$ with the action
of $\gl(V)\times \gl(V)$.

\begin{lemma} \label{proddeglemma}
If $e_B\times e_{B'}\leq _{deg} e_A\times e_{A'}$, then
$e_B\star e_{B'}\leq _{deg} e_A\star e_{A'}$
\end{lemma}
\begin{proof}
Since $e_B\times e_{B'}\subseteq \overline{e_A \times e_{A'}}$
we have $S(B,B')\subseteq \overline{S(A,A')}$. By Corollary
\ref{uniqueorbit}, we have that $\overline{S(A,A')}$ is the orbit
closure of $e_A\star e_{A'}$ and $\overline{S(B,B')}$ is the orbit
closure of $e_B\star e_{B'}$, the lemma follows.
\end{proof}

We have the following key lemma on degeneration and multiplication
in $G(n,r)$.

\begin{lemma} \label{ideallemma}
Let $\sigma \in S_r$,
$e_{B'}\subseteq \mathcal{F}_{\underline d} \times
\mathcal{F}_{\underline e}$ and $e_{B}\subseteq
\mathcal{F}_{\underline f} \times \mathcal{F}_{\underline g}$.
Then $e_{B'}\star (\sigma o_{\underline e,\underline f}) \star e_{B}
\leq_{deg}\sigma o_{\underline d,\underline g}$.
\end{lemma}
\begin{proof}
By Lemma \ref{proddeglemma}, it suffices to consider the case where
$e_B$ and $e_{B'}$ are closed orbits. By Lemma \ref{closedorbit},
we may choose the representation $$\bigoplus^r_{l=1}  N_{j_lk_l},$$
where $k_{l+1}\geq k_{l}$ and $j_{l+1}\geq j_{l}$
for the orbit $e_B$. Similarly, $o_{\underline e,\underline f}$ is the orbit corresponding
to the representation $$\bigoplus^r_{l=1}  N_{i_lj_l},$$ where $i_l\geq i_{l+1}$
by Lemma \ref{openorbit}. Then the coefficient of $\sigma
o_{\underline e,\underline g}$ in
the product $(\sigma o_{\underline e,\underline f})\cdot e_B$ in $S_q(n,r)$
is non-zero, and so $$(\sigma o_{\underline e,\underline f})\star e_B\leq_{deg}
\sigma o_{\underline e,\underline g}.$$  Similarly, $$e_{B'}\star \sigma
o_{\underline e,\underline g} \leq_{deg} \sigma o_{\underline d,\underline g}.$$
By Lemma \ref{proddeglemma}, $$e_{B'}\star \sigma o_{\underline e,\underline f}
\star e_B \leq_{deg}e_{B'}\star \sigma o_{\underline e,\underline g}\leq_{deg} 
o_{\underline d, \underline g},$$
as required.
\end{proof}

\begin{corollary}
Let $\sigma \in S_r$,
$e_{B'}\subseteq \mathcal{F}_{\underline d} \times \mathcal{F}_{\underline e}$
and $e_{B}\subseteq\mathcal{F}_{\underline f}
\times \mathcal{F}_{\underline g}$. Then $e_{B'}\star
(\sigma k_{\underline e,\underline f}) \star e_{B}
\leq_{deg}\sigma k_{\underline d,\underline g}$.
\end{corollary}
\begin{proof}
The corollary follows from the previous lemma since
$\sigma k_{\underline d,\underline g}=\sigma\iota o_{\underline d, \underline e}$,
where $\iota(i)=n-i+1$.
\end{proof}

\begin{corollary} \label{ideallemmacor}
Let $e_{B'}\subseteq \mathcal{F}_{\underline d} \times \mathcal{F}_{\underline e}$ and
$e_{B}\subseteq \mathcal{F}_{\underline f}\times \mathcal{F}_ {\underline g}$.
Then $e_{B'}\star o_{\underline e,\underline f} \star e_{B} =
o_{\underline d,\underline g}$. In particular, $o_{\underline d,\underline e}\star o_{\underline e,\underline f}=
o_{\underline d,\underline f}$.
\end{corollary}
\begin{proof}
By the lemma we know that $e_{B'}\star o_{\underline e,\underline f} \star e_{B}
\leq_{deg} o_{\underline d,\underline g}$. Since $o_{\underline d,\underline g}$
is the unique dense open orbit in $\mathcal{F}_d\times \mathcal{F}_g$, 
the equality follows.
\end{proof}

\begin{corollary}
$M(n,r)$ is a two-sided ideal in $G(n,r)$.
\end{corollary}
\begin{proof}
The previous corollary shows that the $\mathbb{Z}$-submodule
$M(n,r)\subseteq G(n,r)$ is closed under multiplication from both
sides with elements from $G(n,r)$, and so it is a two-sided ideal.
\end{proof}

We will now show that $M(n,r)$ is a direct factor of $G(n,r)$. 

\begin{lemma} \label{lemddd}
$\{o_{\underline d}\}_{\underline d}\bigcup
\{k_{\underline d}-o_{\underline d}\}_{\underline d}$ is a set of pairwise orthogonal
idempotents in $G(n,r)$.
\end{lemma}
\begin{proof}
By Corollary \ref{ideallemmacor}, $(o_{\underline d})^2=
o_{\underline d}$, $(k_{\underline d}-o_{\underline d})o_{\underline d}=
o_{\underline d}-o_{\underline d}=0$, $o_{\underline d}(k_{\underline d}-
o_{\underline d})=o_{\underline d}-o_{\underline d}=0$, and $(k_{\underline d}-
o_{\underline d})^2=(k_{\underline d}-o_{\underline d}-o_{\underline d}+
o_{\underline d})=k_{\underline d}-o_{\underline d}$. All other
orthogonality relations follow from the definition of multiplication
in $S_q(n,r)$.
\end{proof}

Let $M(D(n,r))$ be the algebra of $|D(n,r)|\times |D(n,r)|$-matrices over $\mathbb{Z}$.
Let $$\omega_0:M(n,r)\rightarrow M(D(n,r))$$ be the $\mathbb{Z}$-linear
map where $\omega_0(o_{\underline d,\underline e})=E_{\underline d,\underline e}$
is the $(\underline d,\underline e)$'th elementary matrix in $M(D(n,r))$.

\begin{lemma}
The map $\omega_0:M(n,r)\rightarrow M(D(n,r))$ is a
$\mathbb{Z}$-algebra isomorphism.
\end{lemma}
\begin{proof}
The proposition is an immediate consequence of Lemma \ref{lemddd} and
Corollary \ref{ideallemmacor}.
\end{proof}

We construct a section to the inclusion $M(n,r)\subseteq G(n,r)$.

\begin{lemma}
The map $\omega:G(n,r)\rightarrow M(n,r)$
defined by $$\omega(e_A)=o_{\underline d,\underline e}$$ for all $e_A\subseteq
\mathcal{F}_{\underline d}\times \mathcal{F}_{\underline e}$ is a surjective
$\mathbb{Z}$-algebra homomorphism.
\end{lemma}
\begin{proof}
The map is clearly a surjective $\mathbb{Z}$-module homomorphism.
Let $e_A\subseteq \mathcal{F}_{\underline d}\times \mathcal{F}_{\underline e}$
and $e_B\subseteq \mathcal{F}_{\underline e}\times \mathcal{F}_{\underline f}$. Then
$\omega(e_A\star e_B)$ is the unique open orbit in $\mathcal{F}_{\underline d}\times
\mathcal{F}_{\underline f}$, which is equal to $\omega(e_A)\star \omega(e_B)$, by
Corollary \ref{ideallemmacor}. Moreover, $\omega(1_{G(n,r)})=\omega(\sum_{\underline d}
k_{\underline d})=\sum_{\underline d}o_{\underline d}=1_{M(n,r)}$. This
completes the proof of the lemma.
\end{proof}

We can now prove the main result of this section.

\begin{theorem}
We have an isomorphism of $\mathbb{Z}$-algebras $G(n,r)\rightarrow M(n,r)\times (G(n,r)/M(n,r))$ given by $e_A\mapsto (\omega(e_A), \overline{e_A})$.
\end{theorem}
\begin{proof}
By Corollary \ref{ideallemmacor}, we have
$$M(n,r)=(\sum_{\underline d} o_{\underline d})G(n,r)
(\sum_{\underline d} o_{\underline d}).$$
Now, $1_{G(n,r)} = \sum_{\underline d}k_{\underline d}$, and again by Corollary
\ref{ideallemmacor}, $\sum_{\underline d} o_{\underline d}$ is a central
idempotent in $G(n,r)$. This proves that $M(n,r)$ is a direct factor
in $G(n,r)$, and so the theorem follows.
\end{proof}

Let ${\mathbb{\widetilde{A}}}_n$ denote the preprojective algebra of type
$\mathbb{A}_n$. See \cite{Bill} for the definition and properties
of preprojective algebras.

\begin{corollary}
$S_0(2,r)\simeq M(2,r)\times {\mathbb{\widetilde{A}}}_{r-1}$
\end{corollary}
\begin{proof}
We need to show that $$(\sum k_{\underline d}-o_{\underline d})G(n,r)
(\sum k_{\underline d}-o_{\underline d}) \simeq
{\mathbb{\widetilde{A}}}_{r-1}.$$ First observe that
$(\sum k_{\underline d}-o_{\underline d})G(n,r)
(\sum k_{\underline d}-o_{\underline d})$ is generated
by $e_{1,\underline d}-o_{\underline d-\alpha_2+
\alpha_1,\underline d}$, $f_{1,\underline d}-o_{\underline
d+\alpha_2-\alpha_1,d}$ and $k_{\underline d}-o_{\underline d}$.
A direct computation shows that
the generators satisfy the
preprojective relations. By comparing dimensions we get the
required isomorphism.
\end{proof}

Let $\underline n=n_1+\cdots+n_l$ and $\underline r=r_1+\cdots+r_l$ be
decompositions of $n$ and $r$, respectively, into $l$ parts, where $n_i>0$.
Let $m_i=\sum^{i-1}_{j=1}n_j$, where $m_1=0$.

There is a map $\phi_j$ of flags of length $n_j$ to flags of length $n$ given
by $\phi_j(f)_l=0$ for $l\leq m_j$, $\phi_j(f)_l=f_{l-m_j}$
for $m_j<l\leq m_{j+1}$ and $\phi_j(f)_l=f_{n_j}$ for $l>m_{j+1}$.
The corresponding map on orbits of pairs of flags
$$[f,f']\mapsto [\phi_j(f),\phi_j(f')]$$ is also denoted by $\phi_j$.

Let $$\phi_{\underline n,\underline r}: G(n_1,r_1)\times \cdots \times
G(n_l,r_l)\rightarrow G(n,r)$$ be the $\mathbb{Z}$-linear map defined by
$$(N_1,\cdots,N_l)\mapsto \phi_{1}(N_1)\oplus\cdots \oplus \phi_{l}(N_l).$$

\begin{lemma} \label{above}
The map $$\phi_{\underline n,\underline r}:G(n_1,r_1)\times \cdots \times
G(n_l,r_l)\rightarrow G(n,r)$$ is an injective $\mathbb{Z}$-algebra
homomorphism. Moreover, $\phi_{\underline n,\underline r}(N_1,\cdots,
N_l)\leq_{deg}\phi_{\underline n,\underline r}(N'_1,\cdots, N'_l)$ if and
only if $N_i\leq_{deg} N'_i$ for all $i$.
\end{lemma}
\begin{proof}
Since $\phi_{\underline n,\underline r}$ is injective on basis elements, it is an
injective $\mathbb{Z}$-linear map.  
By Lemma \ref{Lemma2.2}, in terms of matrices, the map is given by $$\phi_{\underline n,\underline r}(e_{A_1},\cdots,e_{A_l})=
e_{A_1\oplus\cdots\oplus  A_l}.$$
Following Lemma \ref{Lemma3.20Schur}, the map 
$\phi_{\underline n,\underline r}$ preserves multiplication and thus is an injective $\mathbb{Z}$-algebra homomorphism.

Let $$N=\phi_{\underline n,\underline r}(N_1,\cdots, N_l)
\mbox{ and } N'=\phi_{\underline n,\underline r}(N'_1,\cdots, N'_l),$$
and $N\leq_{deg} N'$. We may assume that the degeneration is minimal.
By Lemma \ref{deglemma2}, $N'=(t,s)N$ for a transposition $(t,s)$.
Then the transposition $(t,s)$ must act within one $N_i$, since the off-diagonal
blocks of the  matrices of both $N$ and $N'$ are zero, and so
$$(t,s)N =\phi_{\underline n,\underline r}(N_1,\cdots, N_{i-1}, (t',s')N_i,N_{i+1},\cdots,  N_l)$$
for a transposition $(t',s')$. This shows that $N_i\leq_{deg} N'_i$ for all $i$.

The converse also follows from Lemma \ref{deglemma2}.
\end{proof}

Let $\underline n$, $\underline r$ and $m_i$ be as above.
Let $\underline r_i=d_{m_i+1}+\cdots+d_{m_{i+1}}$, a decomposition
of $r_i$ into $n_i$ parts, and $\underline d = d_1 + \cdots + d_n$,
which is an element in $D(n,r)$.
Let $$o_{(\underline d,\underline n)} =\phi(o_{\underline r_1},
\cdots,o_{\underline r_l}),$$ which is an idempotent, by Lemma \ref{above}. 
Note that $k_{\underline d}=
o_{(\underline d,1+\cdots+1)}$ and that $o_{\underline d}=
o_{(\underline d, n)}$, where $n$ denotes the trivial decomposition
of $n$ into $1$ part.



For a given $\underline d$, there is one idempotent for each decomposition
$\underline n$, and so this produces $2^{n-1}$ idempotents in 
$k_{\underline d}G(n,r)k_{\underline d}$,
if $k_{\underline d}$ is in the interior of
the quiver of $G(n,r)$ viewed as an $(n-1)$-simplex. If $k_{\underline d}$ is
on the boundary, but in the interior of a $t$-simplex, then there are $2^{t}$
idempotents. In particular, on a line we get the two idempotents $k_{\underline d}$
and $o_{\underline d}$, and for the vertices of the simplex we have the unique
idempotent $k_{\underline d}=o_{\underline d}$.

\begin{lemma}
If $o_{(\underline d,\underline n)}\leq_{deg} N$, then
$o_{(\underline d,\underline n)}\star N=N\star
o_{(\underline d,\underline n)}=o_{(\underline d,\underline n)}$.
\end{lemma}
\begin{proof}
We have $o_{(\underline d,\underline n)}=o_{(\underline d,\underline n)}\star
o_{(\underline d,\underline n)}\leq_{deg} N\star o_{(\underline d,\underline n)}
\leq_{deg} k_{\underline d}\star o_{(\underline d,\underline n)} =
o_{(\underline d,\underline n)}$ by Lemma \ref{proddeglemma}. The proof
of the other equality is similar.
\end{proof}

\section{Geometric realisation of $0$-Hecke algebras}

In this section let $n=r$ and $\underline d=1+\cdots+1$. In this case
$\mathcal{F}_{\underline d}$ is the complete flag variety. The
idempotent $k_{\underline d}$ is then the unique interior vertex
in the quiver of $G(n,n)$. Let $H_0(n)=k_{\underline d}
G(n,n)k_{\underline d}$. The $\mathbb{Z}$-algebra $H_0(n)$
is called the $0$-Hecke algebra \cite{Carter,Donkin,Norton,KrobTh},
and we give a proof of this fact in this section.

From the previous section we have $2^{n-1}$ distinct
idempotents $o_{(\underline d,\underline n)}$, one for each
decomposition $\underline n=n_1+\cdots+n_l$ of $n$ with $n_i>0$.

Let $$t_i=(i,i+1)k_{\underline d}.$$ We have
$$t_i=o_{(\underline d,\underline n)},$$ where $n=n_1+\cdots+n_{r-1}$
with $n_i=2$ and $n_j=1$ for $j\neq i$, and so $t_i$ is an idempotent.
Also $$t_{i}=f_{\underline d+\alpha_i-\alpha_{i+1}}\star e_{i,\underline d}=
e_{\underline d-\alpha_i+\alpha_{i+1}}\star f_{i,\underline d}.$$

For a permutation $\sigma$ and a transposition $(i,i+1)$, define
$$(i,i+1)\star \sigma = \left\{\begin{matrix} (i,i+1)\sigma
& \mbox{ if } (i,i+1)\sigma \leq_{deg} \sigma \mbox{ and }\\
\sigma & \mbox{ otherwise. }\end{matrix}\right.$$ Write
$\tau\leq_{\star} \sigma$ if $\tau=(i,i+1)\star \sigma$ for some $i$,
and denote the closure as a partial order also by $\leq_{\star}$.
We clearly have that $\tau \leq_{\star} \sigma$ implies $\tau\leq_{deg}
\sigma$.

\begin{lemma}
$t_{i}\star \sigma k_{\underline d}=((i,i+1)\star \sigma) k_{\underline d}$
\end{lemma}
\begin{proof}
Let $A$ be a matrix such that $e_A=\sigma k_{\underline d}$. The matrix
$A$ is a permutation matrix. We assume that $A_{i,r}=A_{i+1,s}=1$.
By Lemma \ref{Lemma3.20Schur},

$$t_i\star \sigma k_{\underline d} = f_{i,\underline d+\alpha_i-\alpha_{i+1}}
\star e_{i,\underline d}\star e_A = \left\{\begin{matrix}e_A & \mbox{ if } r>s \\ e_{A'} &
\mbox{ if } r<s, \end{matrix}\right.$$
where $A'$ is obtained from $A$ by swapping the $i$'th and $(i+1)$'th rows.
The lemma follows since $e_{A'}=(i,i+1)\sigma k_{\underline d}$ and since
$r<s$ if and only if $(i,i+1)\sigma \leq_{deg} \sigma$.
\end{proof}

Although it can be deduced from a bubble sort algorithm that the $t_i$
generate $H_0(n)$ as an algebra, we will give an explicit construction
of each of the basis element using the multiplication in $H_0(n)$.

\begin{lemma} Suppose $i<j$. Then
$t_{i}\star t_{i+1} \star \cdots \star t_{j-1} =(i,i+1,\cdots,j)k_{\underline d}$,
where $(i,i+1,\cdots,j)$ is a cycle in $S_n$.
\end{lemma}
\begin{proof}
The corollary follows from the previous lemma by induction, since
$(i,i+1,\cdots,j)=(i,i+1)(i+1,\cdots, j)=(i,i+1)\star (i+1,\cdots, j)$.
\end{proof}

Let $\sigma$ be a permutation. Let $$t^\sigma=t^{\sigma,n}\star
\cdots \star t^{\sigma,1}$$ be defined by $$t^{\sigma,1}=
t_{1}\star t_{2}\star \cdots \star t_{\sigma^{-1}(1)-1}$$
and then $$t^{\sigma,i}=t_{i}\star
t_{i+1} \star\cdots \star t_{\tau_{i-1}\sigma^{-1}(i)-1},$$ where
$\tau_{i-1}$ is given by $$\tau_{i-1}k_{\underline d}=t^{\sigma,i-1}
\star\cdots\star t^{\sigma,1}.$$

\begin{theorem} \label{generatorhecke}
With the notation above, $t^\sigma=\sigma k_{\underline d}$. Consequently,
the set of all $t^{\sigma}$ for $\sigma\in S_n$ is a multiplicative basis of $H_0(n)$.
\end{theorem}
\begin{proof}
By the previous lemma, $t^{\sigma,1} = (1,\cdots,\sigma^{-1}(1))k_{\underline d}$.
As a representation $\tau_{1}k_{\underline d}=t^{\sigma,1}$ has the summand
$N_{1,\sigma^{-1}(1)}$, which is fixed by any $t_i$ for $i>1$, and therefore
by $t^{\sigma,i}$ for $i>1$. By induction $\tau_{i}k_{\underline d}$
has the summands $N_{j,\sigma^{-1}(j)}$ for $j=1,\cdots,i$, which are fixed
by $t^{\sigma,j}$ for $j>i$. Therefore $t^\sigma=\sigma K_{\underline d}$.
\end{proof}

We construct the idempotents $o_{(\underline d,\underline n)}$ using the
generators $t_i$. Let $[i,j]$ be an interval in $[1,\cdots,n]$. Define $t^{[i,j]}$
by induction as follows. Let $t^{[i,i]}=k_d$ and $$t^{[i,j]}=
t^{[i+1,j]}\star t_{i}\star\cdots \star t_{j-1}.$$
To each decomposition $\underline n=n_1+\cdots+n_l$ we
have the element $$t^{\underline n}=t^{[m_1+1,m_2]}\star\cdots
\star t^{[m_{l}+1,m_{l+1}]}$$ where $m_i=\sum^{i-1}_{j=1}n_j$ and $m_1=0$.

\begin{corollary}
$t^{\underline n}=O_{(\underline d,\underline n)}$.
\end{corollary}

We end this section by giving a proof of the fact that $H_0(n)$
is the $0$-Hecke algebra given with generators and
relations for instance in \cite{KrobTh}. Now recall the $0$-Hecke algebra, denoted by
$\mathcal{H}_0(n)$, which is an $\mathbb{C}$-algebra generated by
$\mathcal{T}_i$ for $i=1,\cdots,n-1$ with generating relations

\begin{itemize}
\item[i)] $\mathcal{T}_i^2=-\mathcal{T}_i$,
\item[ii)] $\mathcal{T}_i\mathcal{T}_{i+1}\mathcal{T}_i=
\mathcal{T}_{i+1}\mathcal{T}_i\mathcal{T}_{i+1}$ and
\item[iii)] $\mathcal{T}_i\mathcal{T}_j=\mathcal{T}_j
\mathcal{T}_i$ for $|i-j|>1$.
\end{itemize}

The algebra $\mathcal{H}_0(n)$ is a specialisation of the $q$-Hecke algebra
at $q=0$ and has dimension $n!$.

\begin{theorem}
$H_0(n)\otimes_\mathbb{Z}\mathbb{C}\simeq \mathcal{H}_0(n)$
\end{theorem}
\begin{proof}
Let $$h:\mathcal{H}_0(n)\rightarrow H_0(n)\otimes_\mathbb{Z}\mathbb{C}$$ be given by
$h(\mathcal{T}_i)=-t_i$. A direct computation in $H_0(n)$ shows that $-t_i$
satisfy the $0$-Hecke relations i), ii) and iii) above, so the map is well defined.
The two algebras have the same dimension over $\mathbb{C}$, and so it suffices
to know that the map is surjective. But the $t_i$ are generators of $H_0(n)$
by Theorem \ref{generatorhecke} and so the proof is complete.
\end{proof}

\end{document}